\documentclass{sn-jnl}

\usepackage{graphicx}%
\usepackage{multirow}%
\usepackage{amsmath,amssymb,amsfonts}%
\usepackage{amsthm}%
\usepackage{mathrsfs}%
\usepackage[title]{appendix}%
\usepackage{xcolor}%
\usepackage{textcomp}%
\usepackage{manyfoot}%
\usepackage{booktabs}%
\usepackage{algorithm}%
\usepackage{algorithmicx}%
\usepackage{algpseudocode}%
\usepackage{listings}%

\usepackage{epstopdf}
\usepackage[caption=false]{subfig}

\usepackage{natbib}

\usepackage{mathtools}
\mathtoolsset{showonlyrefs=true}

\theoremstyle{plain}
\newtheorem{theorem}{Theorem}[section]
\newtheorem{lemma}[theorem]{Lemma}
\newtheorem{corollary}[theorem]{Corollary}
\newtheorem{proposition}[theorem]{Proposition}
\newtheorem{definition}[theorem]{Definition}
\newtheorem{remark}[theorem]{Remark}

\begin{document}

\title{On some mixtures of the Kies distribution}

\author*[1,2]{\fnm{Tsvetelin} \sur{Zaevski}}\email{t\_s\_zaevski@math.bas.bg,  t\_s\_zaevski@abv.bg}

\author[1,3]{\fnm{Nikolay} \sur{Kyurkchiev}}\email{nkyurk@math.bas.bg}

\affil*[1]{\orgdiv{Institute of Mathematics and Informatics}, \orgname{Bulgarian Academy of Sciences}, \orgaddress{\street{Acad. Georgi Bonchev Str., Block 8}, \city{Sofia}, \postcode{1113},  \country{Bulgaria}}}

\affil[2]{\orgdiv{Faculty of Mathematics and Informatics}, \orgname{ Sofia University ”St. Kliment Ohridski”}, \orgaddress{\street{ 5, James Bourchier Blvd.}, \city{Sofia}, \postcode{1164},  \country{Bulgaria}}}

\affil[3]{\orgdiv{Faculty of Mathematics and Informatics}, \orgname{ University of Plovdiv Paisii Hilendarski}, \orgaddress{\street{236, Bulgaria Blvd.}, \city{Plovdiv}, \postcode{4027},  \country{Bulgaria}}}

\abstract{The purpose of this paper is to explore some mixtures of Kies distributions -- discrete and continuous. The last ones are also known as compound distributions. Some conditions for convergence are established. We study the probabilistic properties of these mixtures. Special attention is taken to the so-called Hausdorff saturation. Several particular cases are considered -- bimodal and multimodal distributions, and mixtures based on binomial, geometric,  exponential, gamma, and beta distributions. Some numerical experiments for real-life tasks are provided.}


\maketitle

\section{Introduction}

Mixing of probability distributions is a powerful tool for enlarging their flexibility and implementability. This idea can be applied to the stochastic processes as well as to the random variables.     The applications of such kind distributions can be found in many real-life areas, such as sociology \cite{maiboroda_miroshnychenko_sugakova_2022}, food industry \cite{NADERI2024115433}, information and communication technologies \cite{https://doi.org/10.1002/adts.202200100,yeleyko2022mixture,LI2023102018,ZHANG2023110352}, engineering \cite{hashempour2022weighted}, nanotechnologies \cite{WANG2024246}, biology \cite{liu2023mixture}, meteorology \cite{silveira2023modelling},   medicine and genetics \cite{https://doi.org/10.1002/sim.9367,https://doi.org/10.1002/wics.1611},   finance \cite{DAMICO2023129335}, economy  and energy industry \cite{DOSSANTOS2024113990,BLASQUES2024105575,YAN2022493},  etc. 

We add a new study devoted to mixtures on the Kies  distribution,   originally defined in  \cite{Kies1958}. It appears as a   fractional-linear transformation $t=\frac{y}{y+1} \Leftrightarrow y=\frac{t}{1-t}$ applied to the Weibull distribution, which implies its wide applicability. An important feature of the Kies distribution is its finite domain -- the interval $\left(0,1\right)$.  If one uses  the  transformation $t=\frac{by+a}{y+1} \Leftrightarrow y=\frac{t-a}{b-t}$, $a<b$, instead of $t=\frac{y}{y+1}$, then the distribution will be stated  on the interval $\left(a,b\right)$ -- see for example  \cite{satheesh2014some,sanku2019moments,kies_corrections}. Several extensions of this distribution are available in the science literature. In  \cite{afify2022power,al2020new,kumar2017exponentiated,kumar2017modified}  the authors propose a power transformation   to define   new families. Some composite distributions  are constructed  in  \cite{zaevski_kyurkchiev_2023} -- see also  \cite{alsubie2021properties,al2021modified}. The distributional properties of the minimum and maximum of several Kies distributions are discussed in \cite{zaevski_kyurkchiev_2024}. Some trigonometric transformations are applied in \cite{zaevski_kyurkchiev_trig_2024}. 

In the present paper, we build new probability  distributions mixing a Kies-style family and  assuming that its parameters  are random variables. We specify the resulting distribution via its cumulative  function defining it as the average of the original ones. We establish some necessary conditions which keep the main characteristics of the initial Kies family. The probabilistic properties of the resulting distributions are obtained. Several special cases are examined in detail -- discrete mixtures (bi- and multi-modal, binomial, geometric) as well as mixtures based on the continuous distributions -- exponential, gamma, and beta. An interesting example that arises is the fact that the standard uniform distribution can be viewed as an exponential Kies mixture.  

Another important task we discuss is the so-called {\it saturation,}  defined as the Hausdorff distance between the cumulative distribution function and a $\Gamma$-shaped curve that connects its endpoints. This definition can be easily extended for distributions stated at a  left-finite domain. The so-established term can be viewed as a measure of the speed of occurrence or as an indicator for a critical point.   We derive a semi-closed form formula for the saturation in the general case and apply it to the particular mixtures mentioned above. In addition, this formula turns to explicit for the exponential mixtures. For some additional studies devoted to the Hausdorf saturation, we refer to \cite{math11224620,zaevski_kyurkchiev_2023,zaevski_kyurkchiev_2024,zaevski_kyurkchiev_trig_2024}.

We apply the derived results  to statistical samples generated from two real-life areas -- the financial industry and the social sphere. The  first example is about the calm and volatile periods for the S\&P 500 index, the second one is about  the unemployment insurance issues. These statistical data exhibit quite different behavior. The density of the first one seems to have an infinitely large initial value, whereas the density of the second one has zero endpoints and one peak. Different kinds of parameters of the Kies distribution and the resulting mixtures can approximate both behaviors.  We calibrate the mentioned above mixtures (bimodal, multimodal, binomial, geometric, exponential, gamma, and beta). The derived results are discussed in detail -- we have to mention that the exponential distribution and its gamma extension  produce very good results.

The paper is structured as follows. We present the base we use later in Section \ref{dist_prop}. The Kies mixtures are defined and examined in Section \ref{kies_mixtures}. The Hausdorff saturation is investigated in Section \ref{hausdorff}. Some particular examples are considered in Section \ref{examples}. The  applications of the Kies mixtures  are discussed in \ref{num_ex}.

\section{Preliminaries}\label{dist_prop}

We shall use the   following notations: a large letter for the cumulative distribution function  (CDF) of a distribution, the over-lined letter  for the complementary cumulative distribution function (CCDF),  the corresponding small letter  for the probability density function (PDF), and  the letter $\psi$ for the moment generating function (MGF). Thus if $F\left(t\right)$ is the CDF, then $\overline F\left(t\right)$, $f\left(t\right)$, and $\psi\left(t\right)$ are the corresponding CCDF, PDF, and MGF, respectively. 

 The  Kies distribution  on the  domain  $\left(0,1\right)$ has  CDF, CCDF, and PDF:

\begin{align}
H\left(t\right)&=1-\exp\left(-\lambda\left(\frac{t}{1-t}\right)^\beta\right)\label{eq_1}\\
\overline H\left(t\right)&=\exp\left(-\lambda\left(\frac{t}{1-t}\right)^\beta\right)\label{eq_1_01_2}\\
h\left(t\right)&=\lambda\beta \frac{t^{\beta-1}}{\left(1-t\right)^{\beta+1}}\exp\left(-\lambda\left(\frac{t}{1-t}\right)^\beta\right).\label{eq_1_01_3}\\
\end{align}

\noindent
The shape of the probability density function is obtained in  proposition 2.1 from \cite{zaevski_kyurkchiev_2023}:

\begin{proposition}\label{prop1}

The  value of the PDF at the right endpoint of the  domain is zero, $h\left(1\right)=0$. Let the function $\alpha\left(t\right)$ for $t\in\left(0,1\right)$ be defined as 

\begin{equation}\label{eq6_13}
\alpha\left(t\right):=\lambda \beta\left(\frac{t}{1-t}\right)^\beta-\left(2 t+\beta-1\right).
\end{equation}

\noindent
The following statements for PDF \eqref{eq_1_01_3} w.r.t. the position of the power  $\beta$ w.r.t.  $1$  hold.

\begin{enumerate}

\item
If $\beta>1$, then  PDF \eqref{eq_1_01_3} is zero in the left domain's endpoint, $h\left(0\right)=0$. Function \eqref{eq6_13} has a unique root for $t\in \left(0,1\right)$, we denote it by $t_2$.  The PDF increases for $t\in \left(0,t_2\right)$ having a maximum for $t=t_2$ and decreases for  $t\in \left(t_2,1\right)$.

\item
If $\beta=1$, then the left limit of the PDF is $h\left(0\right)=\lambda$. If $\lambda\ge 2$, then the PDF is  a decreasing from $\lambda$ to $0$ function. Otherwise, if $\lambda< 2$, then we need  the value $t_2=1-\frac{\lambda}{2}$ -- note that  $t_2\in \left(0,1\right)$. The PDF starts from the value $\lambda$ for $t=0$, increases to a maximum for $t=t_2$, and decreases to zero.

\item
If $\beta<1$, then  $h\left(0\right)=\infty$. The derivative of function \eqref{eq6_13} is

\begin{equation}\label{eq6_14}
\alpha'\left(t\right)=\lambda \beta^2\frac{t^{\beta-1}}{\left(1-t\right)^{\beta+1}}\\-2.
\end{equation}

\noindent
Let  $\overline  t$ be defined as $\overline t :=\frac{1-\beta}{2}$. The PDF is a decreasing function when $\alpha'\left(\overline t\right)\ge 0$. 

Suppose that   $\alpha'\left(\overline t\right)< 0$. In this case,  derivative \eqref{eq6_14} has two roots in the interval $\left(0,1\right)$ -- we denote them  by $\overline t_1$ and $\overline t_2$. If $\alpha\left(\overline t_2\right)\ge 0$, then the PDF decreases  in the whole distribution domain. Otherwise, if $\alpha\left(\overline t_2\right)< 0$, then function \eqref{eq6_13} has two roots in the interval $\left(0,1\right)$ too -- we notate them by $ t_1$ and $ t_2$. The PDF starts from infinity, decreases in the interval $\left(0,t_1\right)$ having a local minimum for $t=t_1$, increases for $t \in \left( t_1,  t_2\right)$ having a local maximum for $t=t_2$, and decreases to zero for $t\in \left(  t_2, 1\right)$.

\end{enumerate}

\end{proposition}

Generally said, Proposition \ref{prop1} shows that the PDF may exhibit several forms. In all cases, the right endpoint is zero. Also, it may start from infinity and decrease to zero or have one local minimum and another local maximum. Second, it may start from a finite point and decrease to zero or first to increase to a peak and then to decrease to zero. And finally, the PDF may start from zero, increase to a maximum, and then decrease.

\section{Kies mixtures}\label{kies_mixtures}

We need the following lemma to define the Kies mixtures.

\begin{lemma}

Let  $\lambda$ and $\beta$ be positive random variables on some probability space $\left(\Omega,\mathcal{F},\mathbb{P}\right)$ and  $H\left(t;\lambda,\beta\right)$ be the CDFs of a Kies distributed family. Then the function $F\left(\cdot\right)$, defined as

\begin{equation}\label{eq_2}
F\left(t\right)=\mathbb{E}\left[H\left(t;\lambda,\beta\right)\right],
\end{equation}

\noindent
is continuous and increases from zero to one.
\end{lemma}

\begin{proof}
Continuity follows from the dominated convergence theorem since $0\le H\left(t;\lambda,\beta\right)\le 1$ for all sample events. As a consequence $F\left(0\right)=0$ and $F\left(1\right)=1$. Function \eqref{eq_2} is increasing since all functions  $H\left(t;\lambda,\beta\right)$ increase.
\end{proof}

\begin{definition}\label{def1}

Let us impose the following conditions on the random variables  $\lambda$ and $\beta$:

\begin{align}
&\mathbb{E}\left[\frac{\beta }{\lambda^{\frac{1}{\beta}}}\left(\frac{\beta+1}{\beta }\right)^{\frac{\beta+1}{\beta}}\right]<\infty\label{eq6_14_1}\\
&\mathbb{E}\left[\lambda\right]<\infty\label{eq6_14_2}\\
&\mathbb{E}\left[\lambda\beta\right]<\infty\label{eq6_14_3}.\\
\end{align}

\noindent
 The mixture distribution is defined by its CDF, $F\left(\cdot\right)$, through formula \eqref{eq_2}.
\end{definition}

\begin{remark}
The importance of these conditions will be seen later when we discuss some particular  mixtures. Note that they are sufficient but not necessary and thus one may impose other requirements which lead to similar results.
\end{remark}

 We need the following lemma before  establishing the result for the mixture PDF.

\begin{lemma}\label{lem2}
Let $a$ and $b$ be positive constants. The function 

\begin{equation}\label{eq_9_1}
g\left(x\right)=x^{b+1}e^{-a x^b}
\end{equation}

\noindent
achieves its maximum at the positive real half-line for  $x=\left(\frac{b+1}{ab}\right)^{\frac{1}{b}}$  and it is 

\begin{equation}\label{eq_9_2}
\left(\frac{b+1}{eab}\right)^{\frac{b+1}{b}}.
\end{equation}

\end{lemma}

\begin{proof}
The proof follows the  presentation of the derivative of function \eqref{eq_9_1}

\begin{equation}\label{eq_9_3}
g'\left(x\right)=x^{b}e^{-a x^b}\left(b+1-abx^b\right).
\end{equation}
\end{proof}

\begin{corollary}\label{cor1}

Let $t\in\left(0,1\right]$. We have for the mixture  CCDF and PDF, $\overline F\left(t\right)$ and $f\left(t\right)$ respectively: 

\begin{align}
\overline F\left(t\right)&=\mathbb{E}\left[\overline H\left(t;\lambda,\beta\right)\right]\label{eq_5_1}\\
f\left(t\right)&=\mathbb{E}\left[h\left(t;\lambda,\beta\right)\right].\label{eq_5}\\
\end{align}

\end{corollary}

\begin{proof}
The equality $\overline F\left(t\right)=\mathbb{E}\left[\overline H\left(t;\lambda,\beta\right)\right]$ is obvious. Next, we prove the statement  for the PDF. We can write

\begin{equation}\label{eq_9}
\begin{split}
 F'\left(t\right)&=\lim\limits_{\epsilon\to 0}{\frac{\mathbb{E}\left[H\left(t+\epsilon;\lambda,\beta\right)\right]-\mathbb{E}\left[H\left(t;\lambda,\beta\right)\right]}{\epsilon}}\\
&=\lim\limits_{\epsilon\to 0}{\mathbb{E}\left[\frac{H\left(t+\epsilon;\lambda,\beta\right)-H\left(t;\lambda,\beta\right)}{\epsilon}\right]}\\
&=\lim\limits_{\epsilon\to 0}{\mathbb{E}\left[h\left(\tau\left(\epsilon\right);\lambda,\beta\right)\right]}\\
\end{split}
\end{equation}

\noindent
for some $\tau\left(\epsilon\right)\in\left(t,t+\epsilon\right)$ due to the  mean value theorem.  Using Lemma \ref{lem2} for $a=\lambda t^\beta$, $b=\beta$, and $x=\frac{1}{1-t}$ we obtain

\begin{equation}\label{eq_8}
\begin{split}
\mathbb{E}\left[\left|h\left(t;\lambda,\beta\right)\right|\right]&=\mathbb{E}\left[\lambda\beta \frac{t^{\beta-1}}{\left(1-t\right)^{\beta+1}}\exp\left(-\lambda\left(\frac{t}{1-t}\right)^\beta\right)\right]\\
&\le \mathbb{E}\left[\lambda \beta t^{\beta-1}\left(\frac{\beta+1}{e\lambda\beta t^\beta}\right)^{\frac{\beta+1}{\beta}}\right]\\
&\le \mathbb{E}\left[\frac{\beta\exp\left(-\frac{\beta+1}{\beta }\right) }{t^2\lambda^{\frac{1}{\beta}}}\left(\frac{\beta+1}{\beta }\right)^{\frac{\beta+1}{\beta}}\right].
\end{split}
\end{equation}

\noindent
Hence $\mathbb{E}\left[\left|h\left(t;\lambda,\beta\right)\right|\right]<\infty$  in a neighborhood of every point from the interval  $t\in\left(0,1\right)$ due to condition \eqref{eq6_14_1}. We can apply now the dominated convergence theorem to obtain the desired result for the PDF in formulas \eqref{eq_5}. 
\end{proof}

\begin{corollary}\label{MGF}

Suppose that the random variable $\beta$ is deterministic. The CCDF of a Kies mixture can be derived as the MGF of the random variable $\lambda$ taken at the point $-\left(\frac{t}{1-t}\right)^\beta$

\begin{equation}\label{eq_8_1}
\overline F\left(t\right)=\psi_\lambda\left(-\left(\frac{t}{1-t}\right)^\beta\right).
\end{equation}

\noindent
As a consequence, the PDF turns to

\begin{equation}\label{eq_8_2}
f\left(t\right)=\beta\frac{t^{\beta-1}}{\left(1-t\right)^{\beta+1}}\psi_\lambda '\left(-\left(\frac{t}{1-t}\right)^\beta\right).
\end{equation}

\end{corollary}

\begin{proof}
The corollary follows formulas \eqref{eq_1_01_2} and \eqref{eq_5_1}.
\end{proof}

Obviously the shape of the mixture PDF is closely related to the PDFs of the Kies family as well as to the  random variables $\lambda$ and $\beta$. Nonetheless, we can prove the following proposition which characterizes the endpoints of the PDF.

\begin{proposition}\label{prop_left}

  The right PDF endpoint is $f\left(1\right)=0$. The left endpoint can be derived via the following alternatives. Note that we shall use the symbol $\mathbb{E}$  for the expectation w.r.t. the measure $\mathbb{P}$.

\begin{enumerate}

\item

If $\mathbb{P}\left(\beta>1\right)=1$, then $f\left(0\right)=0$.

\item
If $\mathbb{P}\left(\beta=1\right)>0$ and $\mathbb{P}\left(\beta<1\right)=0$, then  $f\left(0\right)=\mathbb{Q}\left(\beta=1\right)\mathbb{E}\left[\lambda\right]$. The probability measure $\mathbb{Q}$ is equivalent to $\mathbb{P}$ with Radon-Nikodym derivative 

\begin{equation}\label{eq_10}
\frac{d\mathbb{Q}}{d\mathbb{P}}=\frac{\lambda}{\mathbb{E}\left[\lambda\right]}.
\end{equation}

\noindent
Note that the measure $\mathbb{Q}$ really exists since the random variable $\lambda$ is positive and $\mathbb{E}\left[\lambda\right]<\infty$ due to condition \eqref{eq6_14_2}. 

\item

If $\mathbb{P}\left(\beta<1\right)>0$, then $f\left(0\right)=\infty$.

\end{enumerate}

\end{proposition}

\begin{proof}
The value $f\left(1\right)=0$ can be derived via inequality \eqref{eq_8} and the dominated convergence theorem. Let us turn to the left endpoint. We can rewrite formula \eqref{eq_5} as

\begin{equation}\label{eq_11}
\begin{split}
f\left(t\right)&=\mathbb{E}\left[h\left(t;\lambda,\beta\right)I_{\beta<1}\right]+\mathbb{E}\left[h\left(t;\lambda,\beta\right)I_{\beta=1}\right]+\mathbb{E}\left[h\left(t;\lambda,\beta\right)I_{\beta>1}\right].
\end{split}
\end{equation}

Suppose first that $\mathbb{P}\left(\beta>1\right)=1$. Let us define the measure $\mathbb{L}$ for a fixed $t$ via the Radon-Nikodym derivative 

\begin{equation}\label{eq_12}
\frac{d\mathbb{L}}{d\mathbb{P}}=\frac{h\left(t;\lambda,\beta\right)}{\mathbb{E}\left[h\left(t;\lambda,\beta\right)\right]}.
\end{equation}
 
\noindent
Note that for every $t\in\left(0,1\right]$ the expectation $\mathbb{E}\left[h\left(t;\lambda,\beta\right)\right]$ is finite due to inequality  \eqref{eq_8}. Hence the sum of the first and second  expectations   from formula \eqref{eq_11} can be obtained as

\begin{equation}\label{eq_13}
\begin{split}
\mathbb{E}\left[h\left(t;\lambda,\beta\right)I_{\beta\le 1}\right]=\mathbb{E}\left[h\left(t;\lambda,\beta\right)\right]\mathbb{L}\left(\beta\le 1\right)=0,
\end{split}
\end{equation}

\noindent
because the measures $\mathbb{L}$ and $\mathbb{P}$ are equivalent and $\mathbb{P}\left(\beta\le 1\right)=0$. Inequality \eqref{eq6_14_3} allows us to take  the limit $t\to 0$ in the third  expectation from formula \eqref{eq_11} and using the first statement of proposition \ref{prop1} to derive $f\left(0\right)=0$.

Suppose now that $\mathbb{P}\left(\beta=1\right)>0$ and $\mathbb{P}\left(\beta<1\right)=0$.  Analogously as above we can derive that the values of the first and third expectations from formula \eqref{eq_11} are  zero. We derive the desired by result changing the measure  from $\mathbb{P}$ to $\mathbb{Q}$ in the second expectation.

It left to consider the case $\mathbb{P}\left(\beta<1\right)>0$. The third statement of proposition \ref{prop1} shows that $h\left(t;\lambda\left(\omega\right),\beta\left(\omega\right)\right)$  tends to infinity for $t\to 0$ for every sample event $\omega$ such that $\beta\left(\omega\right)<1$. Let $t\in\left(0,1\right]$ and $M$ be a positive constant. We define the set $\Omega_{t,M}$ as 

\begin{equation}\label{eq_14}
\Omega_{t,M}=\left\{\omega\in\Omega: \left(\beta<1\right) \& \left(h\left(t;\lambda,\beta\right)>M\ \forall\ u< t\right) \right\}.
\end{equation}

\noindent
Hence,

\begin{equation}\label{eq_15}
\mathbb{E}\left[h\left(t;\lambda,\beta\right)I_{\beta< 1}\right]\ge M\mathbb{P}\left(\Omega_{t,M}\right).
\end{equation}

\noindent
Having in mind that $\lim_{t\to 0} \Omega_{t,M}= \left\{\omega:\ \beta<1\right\}$, we see that \eqref{eq_15} leads to

\begin{equation}\label{eq_16}
\lim\limits_{t\to 0}\mathbb{E}\left[h\left(t;\lambda,\beta\right)I_{\beta< 1}\right]\ge M\mathbb{P}\left(\beta<1\right).
\end{equation}

\noindent
We conclude that the third expectation of formula \eqref{eq_11} tends to infinity because  inequality \eqref{eq_16} holds for all constants $M$ and $\mathbb{P}\left(\beta<1\right)>0$. This finishes the proof.
\end{proof}

\begin{remark}
We can define the mixed Kies distribution jointly with the pair $\left(\lambda,\beta\right)$ on the probability space $\left(\Omega_,\mathcal{F},\mathbb{P}\right)$. Let the measure $\mu\left(dx_1,dx_2\right)$ on $\mathbb{R}^+\times \mathbb{R}^+$ be associated with the random variables $\left(\lambda,\beta\right)$ and $\xi$ to stand for the mixture Kies random variable. We  define the triple $\left(\xi,\lambda,\beta\right)$ via the joint distribution

\begin{equation}\label{eq_17}
\mathbb P\left(\xi\in A,\lambda\in{B_1},\beta\in{B_2}\right)=\int\limits_{t\in A}{\int\limits_{x\in B}{h\left(t,x_1,x_2\right)\mu\left(dx_1,dx_2\right)}dt}
\end{equation}

\noindent
for arbitrary subsets $A$ and $B$   of the sets $\left(0,1\right)$ and $\mathbb{R}^+\times \mathbb{R}^+$.    We can easily check that formula \eqref{eq_2} holds:

\begin{equation}\label{eq_18}
\begin{split}
F\left(t\right)&=\mathbb{P}\left(\xi<t\right)\\
&=\int\limits_{0}^t{\int\limits_{x\in\left\{ \mathbb{R}^+\times \mathbb{R}^+\right\}}{h\left(t,x_1,x_2\right)\mu\left(dx_1,dx_2\right)}dt}\\
&=\int\limits_{x\in\left\{ \mathbb{R}^+\times \mathbb{R}^+\right\}}{\left(\int\limits_{0}^t{h\left(t,x_1,x_2\right)}dt\right)\mu\left(dx_1,dx_2\right)}\\
&=\mathbb{E}\left[H\left(t;\lambda,\beta\right)\right].
\end{split}
\end{equation}

\noindent
We have used above Fubini's theorem to interchange the order of integration. Note that the defined in that way  random variable $\xi$ is independent of the pair $\left(\lambda,\beta\right)$ only when $\lambda$ and $\beta$ are constants.
\end{remark}

\section{Saturation in the Hausdorff sense}\label{hausdorff}

The Hausdorff distance can be defined  in the sense of \cite{sendov1990hausdorff}:

\begin{definition}\label{def_hausdorff}
Let us consider the max-norm in $\mathbb{R}^{2}$:  if $A$ and $B$ are the points $A=\left(t_A, x_A\right)$ and   $B=\left(t_B, x_B\right)$, then $\lVert A - B \rVert := \max \left\{\left |t_A - t_B \right|, \left|x_A - x_B\right|\right\}$.  The Hausdorff distance $d\left(g,h\right)$ between two   curves $g$ and $h$ in $\mathbb R^2$ is 

\begin{equation}\label{eqq_1}
d \left( g,h\right) := \max \left\{ \sup_{A\in g}{ \inf_{B\in h} {\lVert A-B\rVert}},\sup_{B\in h} \inf_{A\in g} \lVert A-B\rVert \right\}.
\end{equation}

\end{definition}

The Hausdorff distance can be viewed as the highest  optimal path between the curves.  Next we  define  the {\it saturation} of a distribution:

\begin{definition}\label{sat}
Let $F\left(\cdot\right)$ be the CDF of a distribution stated on the interval  $\left(0,1\right)$. Its saturation   is the Hausdorff distance between the completed graph of $F\left(\cdot\right)$ and the curve consisting of two lines -- one vertical between the points $\left(0,0\right)$ and $\left(0,1\right)$ and another horizontal between the points $\left(0,1\right)$ and $\left(1,1\right)$.
 
\end{definition}

The following corollary holds  for the  saturation. 

\begin{corollary}\label{corr_sat}
The saturation $d$ of a  mixed-Kies random variable   is the unique solution of the equation

\begin{equation}\label{eqq_2_0}
F\left(d\right)+d=1.
\end{equation}

\end{corollary}

\begin{proof}
Equation \eqref{eqq_2_0} is true due to  Definitions \ref{def_hausdorff} and \ref{sat}.  Its root is unique because  $l\left(t\right)=F\left(t\right)+t-1$ is an increasing  continuous function with endpoints  $l\left(0\right)=-1<0$ and $l\left(1\right)=1>0$.
\end{proof}

The next two theorems provide a semi-closed form formula for the saturation.

\begin{theorem}\label{th2}
Let the positive random variable $\tau$ be such that

\begin{equation}\label{eqq_2}
\lambda=\tau\left(\frac{1}{\mathbb{E}\left[e^{-\tau}\right]}-1\right)^\beta.
\end{equation}

\noindent
Then the saturation of a min-Kies distribution can be derived as

\begin{equation}\label{eqq_3}
d=\mathbb{E}\left[e^{-\tau}\right].
\end{equation}

\noindent
Note that $0<\mathbb{E}\left[e^{-\tau}\right]<1$.
\end{theorem}

\begin{proof}
Suppose that formula \eqref{eqq_2} holds. Using equations \eqref{eq_1} and \eqref{eq_2} and having in mind $0<\mathbb{E}\left[e^{-\tau}\right]<1$ we obtain

\begin{equation}\label{eqq_50}
\begin{split}
1-F\left(\mathbb{E}\left[e^{-\tau}\right]\right)&=\mathbb{E}\left[1-H\left(\mathbb{E}\left[e^{-\tau}\right];\lambda,\beta\right)\right]\\
&=\mathbb{E}\left[\exp\left\{-\lambda\left(\frac{\mathbb{E}\left[e^{-\tau}\right]}{1-\mathbb{E}\left[e^{-\tau}\right]}\right)^\beta\right\}\right]\\
&=\mathbb{E}\left[\exp\left\{-\tau\left(\frac{1}{\mathbb{E}\left[e^{-\tau}\right]}-1\right)^\beta\left(\frac{\mathbb{E}\left[e^{-\tau}\right]}{1-\mathbb{E}\left[e^{-\tau}\right]}\right)^\beta\right\}\right]\\
&=\mathbb{E}\left[e^{-\tau}\right].
\end{split}
\end{equation}

\noindent
We have to use Corollary \ref{corr_sat} to finish the proof.
\end{proof}

\begin{theorem}\label{th3}
The random variable $\tau$ from Theorem \ref{th2} that satisfies condition \eqref{eqq_2} exists and it is unique.
\end{theorem}

\begin{proof}
Suppose that there exist two random variables $\tau_1$ and $\tau_2$ satisfying equation \eqref{eqq_2}. Hence,

\begin{equation}\label{eqq_51}
\tau_1\left(\frac{1}{\mathbb{E}\left[e^{-\tau_1}\right]}-1\right)^\beta=\tau_2\left(\frac{1}{\mathbb{E}\left[e^{-\tau_2}\right]}-1\right)^\beta,
\end{equation}

\noindent
which can be rewritten as

\begin{equation}\label{eqq_52}
\left(\frac{\tau_1}{\tau_2}\right)^{\frac{1}{\beta}}=\frac{\mathbb{E}\left[e^{-\tau_1}\right]}{1-\mathbb{E}\left[e^{-\tau_1}\right]}\frac{1-\mathbb{E}\left[e^{-\tau_2}\right]}{\mathbb{E}\left[e^{-\tau_2}\right]}.
\end{equation}

\noindent
We can see that the right-hand-side of equation \eqref{eqq_52} is a determinist constant. Therefore the ratio $\frac{\tau_1}{\tau_2}$ can be presented as 

\begin{equation}\label{eqq_53}
\frac{\tau_1}{\tau_2}=c^\beta
\end{equation}

\noindent
for some  constant $c$. Combining equations \eqref{eqq_52} and \eqref{eqq_53} we derive

\begin{equation}\label{eqq_54}
\begin{split}
c&=\frac{\mathbb{E}\left[e^{-c^\beta\tau_2}\right]}{1-\mathbb{E}\left[e^{-c^\beta\tau_2}\right]}\frac{1-\mathbb{E}\left[e^{-\tau_2}\right]}{\mathbb{E}\left[e^{-\tau_2}\right]}\\
&=\left(\frac{1}{1-\mathbb{E}\left[e^{-c^\beta\tau_2}\right]}-1\right)\frac{1-\mathbb{E}\left[e^{-\tau_2}\right]}{\mathbb{E}\left[e^{-\tau_2}\right]}.
\end{split}
\end{equation}

\noindent
The right-hand-side of equation \eqref{eqq_54} is a decreasing function w.r.t. variable $c$. Hence, if we consider \eqref{eqq_54} as an equation w.r.t. $c$, then it has at most one solution. We can easily check that this root is $c=1$, which means $\tau_1=\tau_2$.

We turn to the existence task. Let the function $\gamma\left(x\right)$ be defined as

\begin{equation}\label{eqq_55}
\gamma\left(x\right)=x\left[\frac{1}{\mathbb{E}\left[e^{-\lambda x^\beta}\right]}-1\right].
\end{equation}

\noindent
Note that this function increases from zero to infinity in the interval $x\in\left[0,\infty\right)$. Hence, the equation $\gamma\left(x\right)=1$ has a unique solution -- we denote it by $\overline x$. We shall show that the random variable

\begin{equation}\label{eqq_56}
\tau=\lambda\overline x^\beta
\end{equation}

\noindent
satisfies equation \eqref{eqq_2}. Using the equality $\gamma\left(\overline x\right)=1$, we derive

\begin{equation}\label{eqq_57}
\begin{split}
\left(\frac{\lambda}{\tau}\right)^{\frac{1}{\beta}}=\frac{1}{\overline x}=\frac{1-\mathbb{E}\left[e^{-\lambda \overline x^\beta}\right]}{\mathbb{E}\left[e^{-\lambda \overline x^\beta}\right]}=\frac{1-\mathbb{E}\left[e^{-\tau}\right]}{\mathbb{E}\left[e^{-\tau}\right]},
\end{split}
\end{equation}

\noindent
which is equivalent to equation \eqref{eqq_2}. This finishes the proof. 
\end{proof}

The following theorem is an immediate corollary of Theorems \ref{th2} and \ref{th3} and gives an approach for deriving the saturation of the mixed-Kies distributions.

\begin{theorem}\label{th4}
Let $\overline x$ be the solution of the equation $\gamma\left(x\right)=1$, where the function $\gamma\left(x\right)$ is defined by formula \eqref{eqq_55}. Note that this root is unique in the interval $x\in\left[0,\infty\right)$. Then the saturation can be obtained via formula \eqref{eqq_3}, where $\tau$ is defined by equation \eqref{eqq_56}. Combining equations $\gamma\left(x\right)=1$, \eqref{eqq_3}, and \eqref{eqq_56}, we see that the saturation can be derived also as

\begin{equation}\label{eqq_57_1}
d=\frac{\overline x}{\overline x+1}.
\end{equation}
\end{theorem}

Based on Theorem \ref{th4}, we can establish the following algorithm for deriving the saturation.

\begin{algorithm}\label{alg}
      
We can obtain the saturation through three steps.
\begin{enumerate}
\item
We derive $\overline x$ as the solution of  $\gamma\left(x\right)=1$, where the function $\gamma\left(x\right)$ is given by formula \eqref{eqq_55}.

\item
We obtain $\tau$ via equation \eqref{eqq_56}.

\item
We derive the saturation through formula \eqref{eqq_3} or equivalently through  \eqref{eqq_57_1}.
\end{enumerate}
\end{algorithm}

\section{Some examples}\label{examples}

 We can devise the set of all mixtures introduced  by Definition \ref{def1} into two main classes -- discrete and continuous. We shall consider separately several  examples of both kinds.

\subsection{Discrete mixtures}

Suppose that we have $n\le \infty$ original Kies-distributions -- if $n=\infty$ we want the set of these distributions to be countable. Note that  conditions \eqref{eq6_14_1}-\eqref{eq6_14_3} are satisfied when $n<\infty$. The possible values of the random variables $\lambda$ and $\beta$ shall be denoted by $\lambda_i$ and $\beta_i$, $i=1,2,...,n$. The probabilities of these values to happen are denoted by $p_i>0$. We can write PDF, CDF, and CCDF in the used above terms as 

\begin{equation}\label{eqq_58}
\begin{split}
f\left(t\right)&=\sum\limits_{i=1}^{n}{p_i h\left(t;\lambda_i,\beta_i\right)}\\
F\left(t\right)&=\sum\limits_{i=1}^{n}{p_i H\left(t;\lambda_i,\beta_i\right)}\\
\overline F\left(t\right)&=\sum\limits_{i=1}^{n}{p_i \overline H\left(t;\lambda_i,\beta_i\right)},
\end{split}
\end{equation}

\noindent
where the functions $h\left(t;\lambda_i,\beta_i\right)$, $H\left(t;\lambda_i,\beta_i\right)$, and $\overline H\left(t;\lambda_i,\beta_i\right)$ are given by equations \eqref{eq_1}-\eqref{eq_1_01_3}. Proposition \ref{prop_left} leads to the following results.

\begin{corollary}\label{cor_10}

Let the set $\left\{1,2,...,n\right\}$ be devised into the subsets $A_1$, $A_2$, and $A_3$ such that (i) if $i\in A_1$, then $\beta_i<1$; (ii) if $i\in A_2$, then $\beta_i=1$; (iii) if $i\in A_3$, then $\beta_i>1$. The following statements hold:

\begin{enumerate}

\item

If $A_1\equiv A_2\equiv \emptyset$, then $f\left(0\right)=0$.

\item
If $A_1\equiv \emptyset$ but $ A_2\neq \emptyset$, then  $f\left(t\right)=\sum\limits_{i\in A_2}{p_i \lambda_i}$.

\item

If $ A_1\neq \emptyset$, then $f\left(0\right)=\infty$.

\end{enumerate}

\end{corollary}

\subsubsection{Bimodal distribution}

Let the parameter $\beta$ be deterministic and larger than one, say $\beta=2$, whereas  the  parameter $\lambda$ achieves two values, say  $\lambda\in\left\{0.1,2\right\}$,   with probabilities  $p_1=p_2=0.5$ or $p_1=0.25$ and $p_2=0.75$. Proposition \ref{prop1} shows that the PDFs of the original Kies distributions are zero in the domain's endpoints and have a unique maximum. The mixture distribution exhibits bi-modality -- all PDFs can be seen in Figure \ref{peaks_2}. Of course, some parameter values may lead  to an unimodal distributions -- for example, the random variable  $\lambda\in\left\{1,2\right\}$, together with the same values for the rest parameters, leads to a such PDF, see Figure \ref{peaks_1}.

Another example for bi-modality can be seen at Figure \ref{peaks_2_fin} -- the assumed parameters are $\lambda\in\left\{2,0.5\right\}$ and $\beta\in\left\{2,1\right\}$. The probabilities are again $p_1=p_2=0.5$ or $p_1=0.25$ and $p_2=0.75$. The main difference with the previous examples is that one of the $\beta$-values  is one and another is larger than one. Corollary \ref{cor_10}, the second statement, explains why the left endpoint of the PDF is larger than zero but finite.  Something more, its value is $p_2\lambda_2$ -- $\frac{1}{4}$ and $\frac{3}{8}$ for both possibilities for $p_1$ and $p_2$, respectively. 

Finally, we present a bimodal distribution with infinite left endpoint -- see Figure \ref{peaks_2_infin}. The chosen parameters are  $\lambda\in\left\{2,0.5\right\}$ and $\beta\in\left\{0.2,2\right\}$; the probabilities are the same as above. The third statement of Corollary \ref{cor_10} confirms that the mixture's left endpoints are the infinity since $\beta_1<1$.

Let us discuss the Hausdorff saturation defined in Section  \ref{hausdorff}. We shall follow Algorithm \ref{alg}. We first have to derive the solution of equation $\gamma\left(x\right)=1$, which now turns to

\begin{equation}\label{eqq_58_10}
x\left(\frac{1}{\sum\limits_{j=1}^n{p_j e^{-\lambda_j x^{\beta_j}}}}-1\right)=1.
\end{equation}

\noindent
Estimating $n=2$, $\beta=2$, $p_1=p_2=0.5$,  $\lambda_1=0.1$, and $\lambda_2=2$, we derive $\overline x=1.0338$. Hence, the values of $\tau=\lambda \overline x^\beta$ are $\tau_1=0.1069$ and $\tau_2=2.1374$ and the saturation is $d=0.5083$. We can easily check that equation \eqref{eqq_2_0} holds.

For the second case ($p_1=0.25$, $p_2=0.75$), we derive $\overline x=0.7986$,  $\tau_1=0.0638$,  $\tau_2=1.2755$, and  $d=0.4440$. The CDF together with the saturation can be viewed in Figure \ref{bim_cdf}. The saturation is indicated by a red circle. Note that the red lines form a square with vertices $\left(0,1-d\right)$, $\left(d,1-d\right)$, $\left(d,1\right)$, and $\left(0,1\right)$.

\subsubsection{Multimodal distributions}

We present a multimodal Kies mixture based on four original Kies distributions  in Figure \ref{peaks_4}. The used parameters are $\lambda\in\left\{0.1,0.5,5,10\right\}$ and $\beta=2$. The probabilities are assumed to be equal, i.e. $p_1=p_2=p_3=p_4=0.25$.  The first statement of Corollary\ref{cor_10} shows  that the mixture left endpoint is zero since $\beta>1$.

Applying Algorithm \ref{alg} and having in mind equation \eqref{eqq_58_10}, we derive $\overline x=0.7700$ and  $\tau\in\left\{0.0593,    0.2965,    2.9646,    5.9291\right\}$. Thus   the saturation is $d=0.4350$.

\subsubsection{Binomial underlying distribution}

Let  the random variable $\lambda-1$  be binomial distributed with parameters $\left(n,p\right)$ and $ \beta$ be a constant. Note that $\lambda$ is positive. We have for the probabilities $p_i$, $ i=1,2,...,n+1$:

\begin{equation}\label{eqq_58_1}
p_i=\mathbb{P}\left(\lambda=i\right)=\binom{n}{i-1}p^{i-1}\left(1-p\right)^{n+1-i}.
\end{equation}

\noindent
We turn to deriving the CCDF. We cannot use directly Corollary \ref{MGF} because we know the distribution of the random variable $\lambda-1$, not $\lambda$. However, we have 

\begin{equation}\label{eqq_59}
\begin{split}
\overline F\left(t\right)&=\sum\limits_{i=1}^{n+1}{p_i \overline H\left(t;\lambda_i,\beta\right)}\\
&=\sum\limits_{i=1}^{n+1}{p_i\exp\left(-i\left(\frac{t}{1-t}\right)^\beta\right)}\\
&=\exp\left(-\left(\frac{t}{1-t}\right)^\beta\right)\sum\limits_{i=0}^{n}{p_i\exp\left(-\left(i-1\right)\left(\frac{t}{1-t}\right)^\beta\right)}.\\
\end{split}
\end{equation}

\noindent
We can recognize in the sum above the MGF of the binomial distribution applied to the term $-\left(\frac{t}{1-t}\right)^\beta$. Hence, if we denote this function by $\psi\left(x\right)$,

\begin{equation}\label{eqq_59_1}
\psi\left(x\right)=\left(1-p+pe^x\right)^n,
\end{equation}

\noindent
 then we derive

\begin{equation}\label{eqq_60}
\begin{split}
\overline F\left(t\right)&=\exp\left(-\left(\frac{t}{1-t}\right)^\beta\right)\psi\left(-\left(\frac{t}{1-t}\right)^\beta\right)\\
&=\exp\left(-\left(\frac{t}{1-t}\right)^\beta\right)\left(1-p+p\exp\left(-\left(\frac{t}{1-t}\right)^\beta\right)\right)^n.
\end{split}
\end{equation}

\noindent
Having in mind that the derivative of the function $\left(\frac{t}{1-t}\right)$ is $\frac{1}{\left(1-t\right)^2}$ we obtain for the binomial  Kies mixture:

\begin{equation}\label{eqq_61}
\begin{split}
f\left(t\right)&=\beta\frac{t^{\beta-1}}{\left(1-t\right)^{\beta+1}}\exp\left(-\left(\frac{t}{1-t}\right)^\beta\right)\left(1-p+p\exp\left(-\left(\frac{t}{1-t}\right)^\beta\right)\right)^n\\
&+np\beta\frac{t^{\beta-1}}{\left(1-t\right)^{\beta+1}}\exp\left(-2\left(\frac{t}{1-t}\right)^\beta\right)\left(1-p+p\exp\left(-\left(\frac{t}{1-t}\right)^\beta\right)\right)^{n-1}\\
&=\beta\frac{t^{\beta-1}}{\left(1-t\right)^{\beta+1}}\exp\left(-\left(\frac{t}{1-t}\right)^\beta\right)\left(1-p+p\exp\left(-\left(\frac{t}{1-t}\right)^\beta\right)\right)^{n-1}\times\\
&\times\left(1-p+p\left(n+1\right)\exp\left(-\left(\frac{t}{1-t}\right)^\beta\right)\right).
\end{split}
\end{equation}

\noindent
We present at Figure \ref{binomial}  the PDFs of the Kies mixtures when $\beta=2$, $n\in\left\{10,50\right\}$, and $p\in\left\{0.5,0.25\right\}$.

We use again  Algorithm \ref{alg}   to obtain the Hausdorff saturation $d$. Remind that $\lambda=1,2,...,11$ when $n=10$. In the case $p=0.25$, we derive for the values of  $p_i$, $\overline x$, and $\tau$:

\begin{equation}\label{eqq_61_1}
\begin{split}
&p\in\{0.0563 , 0.1877, 0.2816, 0.2503, 0.1460, 0.0584, 0.0162, 0.0031, 0.0004, \\
&2.8610\times 10^{-5}, 9.5367\times 10^{-7} \}\\
&\overline x=0.5632\\
&\tau\in\{0.3172, 0.6344, 0.9515, 1.2687, 1.5859, 1.9031, 2.2203, 2.5374,\\
& 2.8546, 3.1718, 3.4890 \}.
\end{split}
\end{equation}

\noindent
We obtain the saturation  as $d=0.3603$ via formula \eqref{eqq_57_1}. Alternatively, we can  view the expectation in formula \eqref{eqq_55} as the MGF, i.e. 

\begin{equation}\label{eqq_55_1}
\mathbb{E}\left[e^{-\lambda x^\beta}\right]=e^{-x^2}\left(1-p+pe^{-x^2}\right)^n
\end{equation}

\noindent
Remind that  $\lambda$ is one plus a binomial distributed random variable. Thus, the equation $\gamma\left(x\right)=1$ turns to 

\begin{equation}\label{eqq_55_2}
x\left[\frac{e^{x^2}}{\left(1-p+pe^{-x^2}\right)^n}-1\right]=1.
\end{equation}

\noindent
Following the same approach we derive for the saturation: $\left\{\overline x=0.4510,d=0.3108\right\}$ when $\left\{\beta=2,n=10,p=0.5\right\}$; $\left\{\overline x=0.3279,d=0.2470\right\}$ when $\left\{\beta=2,n=50,p=0.25\right\}$; and $\left\{\overline x=0.2506,d=0.2004\right\}$ when $\left\{\beta=2,n=50,p=0.5\right\}$. The CDF together with the Hausdorff saturation are presented in Figure \ref{bin_cdf}.

\subsubsection{Geometric underlying distribution}

We investigate now a  Kies mixture based on a geometric  distribution with parameter $p$ on the support $\left\{1,2,...\right\}$ for the random variable $\lambda$. This is a mixture between an infinite number of original Kies distributions. Note that conditions \eqref{eq6_14_1}-\eqref{eq6_14_2} are satisfied because $\mathbb{E}\left[\lambda^{-\frac{1}{\beta}}\right]<1$.   The parameter $\beta$ is again assumed to be deterministic. The probabilities $p_i$ for $ i=1,2,...$  are

\begin{equation}\label{eqq_62}
p_i=\mathbb{P}\left(\lambda=i\right)=p\left(1-p\right)^{i-1}.
\end{equation}

\noindent
We can obtain the mixture CCDF through Corollary \ref{MGF} as the MGF of the geometric distributions $\psi\left(x\right)$ taken at the point $-\left(\frac{t}{1-t}\right)^\beta$. Having in mind that this function is 

\begin{equation}\label{eqq_63}
\psi\left(x\right)=\frac{pe^x}{1-\left(1-p\right)e^x},
\end{equation}

\noindent
 we conclude

\begin{equation}\label{eqq_64}
\overline F\left(t\right)=\frac{p}{\exp\left(\left(\frac{t}{1-t}\right)^\beta\right)+p-1}.
\end{equation}

\noindent
Differentiating we derive the PDF as 

\begin{equation}\label{eqq_65}
f\left(t\right)=\frac{p\beta\exp\left(\left(\frac{t}{1-t}\right)^\beta\right)t^{\beta-1}}{\left(\exp\left(\left(\frac{t}{1-t}\right)^\beta\right)+p-1\right)^2\left(1-t\right)^{\beta+1}}.
\end{equation}

\noindent
We present in Figure \ref{geometric} the PDFs of the obtained mixtures considering the parameter $p$ in the set $\left\{0.25,0.5,0.75\right\}$. We assume also that $\beta=2$. The saturations can be derived using Algorithm \ref{alg}. We again view the expectation in  function $\gamma\left(\cdot\right)$, given by \eqref{eqq_55}, as  MGF \eqref{eqq_63} taken at the point $-x^\beta$. Thus, equation $\gamma\left(x\right)=1$ turns to

\begin{equation}\label{eqq_65_1}
x\left(e^{x^\beta}-1\right)=p.
\end{equation}

\noindent
Solving this equation we derive for its root $\overline x$ and the related saturation:  $\left\{\overline x=0.5931,d=0.3723\right\}$ when $\left\{\beta=2,p=0.25\right\}$; $\left\{\overline x=0.7245,d=0.4201\right\}$ when $\left\{\beta=2,p=0.5\right\}$; and $\left\{\overline x=0.8097,d=0.4474\right\}$ when $\left\{\beta=2,p=0.75\right\}$. We present the CDF in the case $\left\{\beta=2,p=0.25\right\}$ in Figure \ref{geom_cdf} -- there can be seen the saturation as a red point too.

\subsection{Continuous distributions}\label{cont}

We consider now several mixtures based on the assumption that the parameter $\lambda$ follows some continuous distribution with support on the positive real half-line. The random variable  $\beta$ is assumed to be deterministic.

\subsubsection{Exponential distribution}

Let  $\beta>1$ and  the random variable $\lambda$ be exponentially distributed with intensity $\theta$.  Thus its PDF is $p\left(x\right)=\theta e^{-\theta x}$. The MGF is defined for $x<\theta$ and it is 

\begin{equation}\label{eqq_66}
\psi\left(x\right)=\frac{\theta}{\theta-x}.
\end{equation}

\noindent
Applying Corollary \ref{MGF} we see that the CCDF of the  Kies mixture can be obtain after the substitution $x=-\left(\frac{t}{1-t}\right)^\beta$. Note that the condition $x<\theta$ is satisfied. Hence,

\begin{equation}\label{eqq_67}
\overline F\left(t\right)=\frac{\theta\left(1-t\right)^\beta}{\theta\left(1-t\right)^\beta+t^\beta}.
\end{equation}

\noindent
Differentiating, we derive the PDF of the exponential mixture 

\begin{equation}\label{eqq_68}
f\left(t\right)=\frac{\theta\beta t^{\beta-1}\left(1-t\right)^{\beta-1}}{\left(\theta\left(1-t\right)^{\beta}+t^{\beta}\right)^2}.
\end{equation}

\noindent
Several PDFs can be seen in Figure \ref{expo} -- the intensity parameter $\theta$ is among $\left\{0.5,1,2,5\right\}$. We find the  saturation using  Algorithm \ref{alg}  having in mind that if  $\beta$ is a fixed constant, then the expectation given by formula \eqref{eqq_55} function $\gamma\left(\cdot\right)$ is  MGF \eqref{eqq_66} evaluated at the point $-x^\beta$. In the case of the exponential distribution, the equation $\gamma\left(x\right)=1$ can be solved explicitly, and it leads to

\begin{equation}\label{eqq_68_1}
\overline x=\theta^{\frac{1}{\beta+1}} \Leftrightarrow d=\frac{\theta^{\frac{1}{\beta+1}}}{\theta^{\frac{1}{\beta+1}}+1}.
\end{equation}

\noindent
Thus, we derive $\left\{\overline x=0.7937,d=0.4425\right\}$ when $\left\{\beta=2,\theta=0.5\right\}$; $\left\{\overline x=1,d=0.5\right\}$ when $\left\{\beta=2,\theta=1\right\}$; $\left\{\overline x=1.2599,d=0.5575\right\}$ when $\left\{\beta=2,\theta=2\right\}$;\newline  $\left\{\overline x=1.7100,d=0.6310\right\}$ when $\left\{\beta=2,\theta=5\right\}$. The CDF, together with the saturation, in the case $\left\{\beta=2,\theta=1\right\}$ can be seen in Figure \ref{exp_cdf}.

Let us check conditions \eqref{eq6_14_1}-\eqref{eq6_14_3}. Obviously, requirements \eqref{eq6_14_2} and \eqref{eq6_14_3} hold.  We need to consider the expectation $\mathbb{E}\left[\lambda^{-\frac{1}{\beta}}\right]$ for condition \eqref{eq6_14_1}. It can be written as

\begin{equation}\label{eqq_69}
\mathbb{E}\left[\lambda^{-\frac{1}{\beta}}\right]=\int\limits_0^\infty{x^{-\frac{1}{\beta}}\theta e^{-\theta x}dx}=\theta^{\frac{1}{\beta}}\int\limits_0^\infty{y^{\left(1-\frac{1}{\beta}\right)-1}e^{-y}dy}.
\end{equation}

\noindent
The last integral converges only when $\beta>1$ -- the limit is the gamma function $\Gamma\left(1-\frac{1}{\beta}\right)$. On the opposite, if $\beta\le 1$, integral \eqref{eqq_69} diverges, and hence condition \eqref{eq6_14_1} is not satisfied.

Let us consider now the case $\beta=1$. Function \eqref{eq_2} again defines a distribution. Its CCDF can be derived once again through the MGF and thus it and the PDF turn to 

\begin{equation}\label{eqq_70}
\begin{split}
\overline F\left(t\right)&=\frac{\theta\left(1-t\right)}{\theta\left(1-t\right)+t}\\
f\left(t\right)&=\frac{\theta }{\left(\theta\left(1-t\right)+t\right)^2}.
\end{split}
\end{equation}

\noindent
We can see that Corollary \ref{cor1} and Proposition \ref{prop_left} do not hold. For example, $\mathbb{E}\left[h\left(1,\lambda,\beta\right)\right]=0$, but $f\left(1\right)=\theta$.  We can formulate also the following result

\begin{proposition}
The uniform distribution on the interval $\left(0,1\right)$ can be viewed as a Kies mixture with $\beta=1$ and exponentially distributed $\lambda$ with intensity one.
\end{proposition}

\begin{proof}
We recognize the CCDF and PDF of the uniform distribution in formulas \eqref{eqq_70} when $\theta=1$.
\end{proof}

\subsubsection{Gamma distribution}

Suppose now that the random variable $\lambda$ is gamma distributed with shape and rate parameters $\alpha$ and $\theta$, respectively. Note that this distribution generalizes the exponential one.  The gamma PDF can be written as

\begin{equation}\label{eqq_71}
p\left(x\right)=\frac{\theta^\alpha}{\Gamma\left(\alpha\right)}x^{\alpha-1}e^{-\theta x}.
\end{equation}

\noindent
 Let $\beta$ be again a constant. We shall check first when conditions \eqref{eq6_14_1}-\eqref{eq6_14_3} are satisfied. Obviously, we need to consider the expectation $\mathbb{E}\left[\lambda^{-\frac{1}{\beta}}\right]$. We have

\begin{equation}\label{eqq_72}
\mathbb{E}\left[\lambda^{-\frac{1}{\beta}}\right]=\frac{\theta^\alpha}{\Gamma\left(\alpha\right)}\int\limits_0^\infty{x^{\alpha-\frac{1}{\beta}-1}e^{-\theta x}dx}=\theta^{\frac{1}{\beta}}\int\limits_0^\infty{y^{\alpha-\frac{1}{\beta}-1}e^{-y}dy}.
\end{equation}

\noindent
The integral above converges only when $\beta>\frac{1}{\alpha}$ -- its value is $\Gamma\left(\alpha-\frac{1}{\beta}\right)$. We have for the gamma MGF

\begin{equation}\label{eqq_73}
\psi\left(x\right)=\left(\frac{\theta}{\theta-x}\right)^\alpha
\end{equation}

\noindent
for $x<\theta$. Applying again Corollary \ref{MGF}, we obtain for the mixture CCDF

\begin{equation}\label{eqq_74}
\overline F\left(t\right)=\psi\left(-\left(\frac{t}{1-t}\right)^\beta\right)=\left(\frac{\theta\left(1-t\right)^\beta}{\theta\left(1-t\right)^\beta+t^\beta}\right)^\alpha.
\end{equation}

\noindent
Differentiating, we derive the PDF as 

\begin{equation}\label{eqq_75}
f\left(t\right)=\frac{\alpha\theta^\alpha\beta t^{\beta-1}\left(1-t\right)^{\alpha\beta-1}}{\left(\theta\left(1-t\right)^{\beta}+t^{\beta}\right)^{\alpha+1}}.
\end{equation}

\noindent
We present some PDFs in Figure \ref{gam_2} for $\alpha=2$ and $\theta\in\left\{0.5,1,2,5\right\}$. We again use Algorithm \ref{alg} to derive the saturation. The equation $\gamma\left(x\right)=1$ now can be written as

\begin{equation}\label{eqq_75_1}
x\left(\left(\theta+x^\beta\right)^{\alpha}-\theta^{\alpha}\right)=\theta^{\alpha},
\end{equation}

\noindent
because the expectation in \eqref{eqq_55} is MGF \eqref{eqq_73} evaluated at the point $-\lambda x^\beta$. Hence,  $\left\{\overline x=0.5731,d=0.3643\right\}$ when $\left\{\alpha=2,\beta=2,\theta=0.5\right\}$; $\left\{\overline x=0.7332,d=0.4230\right\}$ when $\left\{\alpha=2,\beta=2,\theta=1\right\}$; $\left\{\overline x=0.9361,d=0.4835\right\}$ when $\left\{\alpha=2,\beta=2,\theta=2\right\}$; and  $\left\{\overline x=1.2894,d=0.5632\right\}$ when $\left\{\alpha=2,\beta=2,\theta=5\right\}$.

Suppose now that $\beta=1$. Formula \eqref{eqq_75} leads to $f\left(0\right)=\frac{\alpha}{\theta}$ in which we recognize the expectation of the gamma distribution, $\mathbb{E}\left[\lambda\right]$. This is in accordance with the second statement of Proposition \ref{prop_left}. Note that $\mathbb{Q}\left(\beta=1\right)=1$. The PDFs for the same values of $\alpha$ and $\theta$ can be viewed in Figure \ref{gam_1}. Their initial values are $\frac{\alpha}{\theta}$, particularly $4$, $2$, $0.5$, and $0.4$. The saturations can be derived in the same way as above -- they are $\left\{\overline x=0.4196,d=0.2956\right\}$ when $\left\{\alpha=2,\beta=1,\theta=0.5\right\}$; $\left\{\overline x=0.6180,d=0.3820\right\}$ when $\left\{\alpha=2,\beta=1,\theta=1\right\}$; $\left\{\overline x=0.9032,d=0.4746\right\}$ when $\left\{\alpha=2,\beta=1,\theta=2\right\}$; and  $\left\{\overline x=1.4760,d=0.5961\right\}$ when $\left\{\alpha=2,\beta=1,\theta=5\right\}$.

Finally, we present some PDFs of the third kind w.r.t. the initial value assuming that $\beta=0.7$ -- see Figure \ref{gam_07}. Note that $f\left(0\right)=\infty $ due to the third statement of Proposition \ref{prop_left}. The saturations can be obtained analogously -- $\left\{\overline x=0.3512,d=0.2599\right\}$ when $\left\{\alpha=2,\beta=0.7,\theta=0.5\right\}$; $\left\{\overline x=0.5615,d=0.3596\right\}$ when $\left\{\alpha=2,\beta=0.7,\theta=1\right\}$; $\left\{\overline x=0.8855,d=0.4696\right\}$ when $\left\{\alpha=2,\beta=0.7,\theta=2\right\}$; and  $\left\{\overline x=1.5885,d=0.6137\right\}$ when $\left\{\alpha=2,\beta=0.7,\theta=5\right\}$. The CDF in the case $\left\{\alpha=2,\beta=0.7,\theta=0.5\right\}$ together with the related saturation can be seen in figure \ref{gam_cdf}.

\subsubsection{Beta distribution}

Let us assume now  that the random variable $\lambda$ is beta distributed with parameters $\alpha$ and $\theta$, i.e. its PDF and MGF are

\begin{equation}\label{eqq_76}
\begin{split}
p\left(x\right)&=\frac{x^{\alpha-1}\left(1-x\right)^{\theta -1}}{B\left(\alpha,\theta\right)dx}\\
\psi\left(x\right)&= {}_1 F_1\left(\alpha,\alpha+\theta,x\right),
\end{split}
\end{equation}

\noindent
where ${}_1 F_1\left(\cdot,\cdot,\cdot\right)$ is the confluent hypergeometric function of the first kind and the beta function is defined as the ratio

\begin{equation}\label{eqq_77}
B\left(\alpha,\theta\right)=\frac{\Gamma\left(\alpha\right)\Gamma\left(\theta\right)}{\Gamma\left(\alpha+\theta\right)}.
\end{equation}

\noindent
We shall check when condition \eqref{eq6_14_1} is satisfied. We have

\begin{equation}\label{eqq_78}
\mathbb{E}\left[\lambda^{-\frac{1}{\beta}}\right]=\int\limits_0^\infty{x^{-\frac{1}{\beta}}\frac{x^{\alpha-1}\left(1-x\right)^{\beta -1}}{B\left(\alpha,\beta\right)}dx}=\frac{1}{B\left(\alpha,\beta\right)}\int\limits_0^\infty{x^{\alpha-\frac{1}{\beta}-1}\left(1-x\right)^{\beta -1}dx}.
\end{equation}

\noindent
The integral above converges when $\beta>\frac{1}{\alpha}$. Applying  Corollary \ref{MGF}, we obtain the CCDF of the Kies-beta-mixture 

\begin{equation}\label{eqq_79}
\overline F\left(t\right)=\psi\left(-\left(\frac{t}{1-t}\right)^\beta\right)={}_1 F_1\left(\alpha,\alpha+\theta,-\left(\frac{t}{1-t}\right)^\beta\right).
\end{equation}

\noindent
Differentiating in equation \eqref{eqq_79}  and using the formula for the confluent hypergeometric function's  derivative  -- see for example \cite{gradshteyn2014table}, page 1023, formula 9.213 -- we derive for the PDF of the mixture

\begin{equation}\label{eqq_80}
f\left(t\right)=\frac{\alpha\beta}{\alpha+\theta}{}_1 F_1\left(\alpha+1,\alpha+\theta+1,-\left(\frac{t}{1-t}\right)^\beta\right)\frac{t^{\beta-1}}{\left(1-t\right)^{\beta+1}}.
\end{equation}

\noindent
Let us check the behavior of PDF \eqref{eqq_80} when $t\to 1$ or equivalently $-\left(\frac{t}{1-t}\right)^\beta\to -\infty$. Having in mind that ${}_1 F_1\left(\alpha+1,\alpha+\theta+1,-x\right)$ tends asymptotically  to $\frac{\Gamma\left(\alpha+\theta+1\right)}{\Gamma\left(\theta\right)}x^{-\left(\alpha+1\right)}$ when $x\to\infty$ -- see \cite{abramowitz1968handbook}, page 508, formula 13.5.1 -- we derive

\begin{equation}\label{eqq_82}
\begin{split}
f\left(1\right)&=\lim\limits_{t\to 1}{\frac{\alpha\beta}{\alpha+\theta}\frac{\Gamma\left(\alpha+\theta+1\right)}{\Gamma\left(\theta\right)}\left(\frac{t}{1-t}\right)^{-\beta\left(\alpha+1\right)}\frac{t^{\beta-1}}{\left(1-t\right)^{\beta+1}}}\\
&=\alpha\beta\frac{\Gamma\left(\alpha+\theta\right)}{\Gamma\left(\theta\right)}\lim\limits_{t\to 1}{\left(1-t\right)^{\alpha\beta-1}}.
\end{split}
\end{equation}

\noindent
We conclude that $f\left(1\right)=\infty$ when $\beta<\frac{1}{\alpha}$; $f\left(1\right)=\alpha\beta\frac{\Gamma\left(\alpha+\theta\right)}{\Gamma\left(\theta\right)}$ when $\beta=\frac{1}{\alpha}$; and  $f\left(1\right)=0$ when $\beta>\frac{1}{\alpha}$. Thus, the right endpoint is zero if condition \eqref{eq6_14_1} is satisfied due to equation \eqref{eqq_78}.  Having in mind that the value of the confluent hypergeometric function in the zero is one, we  find a confirmation of Proposition \ref{prop_left} in formula \eqref{eqq_80}. Note that if $\beta=1$, then $f\left(0\right)=\frac{\alpha}{\alpha+\theta}$, which is the expectation of the beta distribution. Some PDFs can be seen in Figure \ref{bet} -- the considered parameters are $\alpha=3$, $\theta=1$, and $\beta\in\left\{0.5,1,2\right\}$. Note that the condition $\beta>\frac{1}{\alpha}$ holds.  We derive the saturations  through Algorithm \ref{alg}. Now, the equation $\gamma\left(x\right)=1$ can be written as

\begin{equation}\label{eqq_82_1}
x\left[\frac{1}{{}_1 F_1\left(\alpha,\alpha+\theta,-x^\beta\right)}-1\right]=1.
\end{equation}

\noindent
The derived values are $\left\{\overline x=0.9577,d=0.4892\right\}$ when $\left\{\alpha=3,\beta=0.5,\theta=1\right\}$; $\left\{\overline x=0.9697,d=0.4923\right\}$ when $\left\{\alpha=3,\beta=1,\theta=1\right\}$; $\left\{\overline x=0.9806,d=0.4951\right\}$ when $\left\{\alpha=3,\beta=2,\theta=1\right\}$. The CDF  and the saturation for the last triple  are presented in Figure \ref{bet_cdf}.

\section{An application}\label{num_ex}

We shall check now the benefits, which the proposed mixture gives,  using two empirical samples -- one arising from the financial markets and another for the unemployment insurance issues. These data are considered in \cite{zaevski_kyurkchiev_2023,zaevski_kyurkchiev_2024} too.  We use a modification of  the classical least square errors approach. First, note that we need to scale the empirical data because the considered distributions are stated at the domain $\left(0,1\right)$. Let the total number of observations be denoted by $N$. We devise the domain into $m$ sub-intervals and denote by $N_i$ the number of observations that   fall in the $i$-th one, $i=1,2,...,m$. We derive the empirical PDF values as $l_i^{\mathrm{emp}}=\frac{mN_i}{N}$. We assign these values to the centers of the sub-intervals. We denote by $l_i^{\mathrm{th}}\left(\gamma\right)$ the theoretical PDF values at these points  for a Kies mixture with parameter's set $\gamma$. We define our cost function as

\begin{equation}\label{eq20_1}
L\left(\gamma\right):=\sum\limits_{i=1}^m{\left|\ln\left(l_i^{\mathrm{emp}}+\epsilon\right)-\ln\left(l_i^{\mathrm{th}}\left(\gamma\right)+\epsilon\right)\right|}.
\end{equation}

\noindent
We want to obtain the parameters $\gamma$ which minimize function \eqref{eq20_1}. We introduce the logarithmic correction because some Kies distributions tend to infinity in the left endpoint. The constant $\epsilon$ is necessary because some empirical values can be equal to zero which will lead to the minus infinity for the logarithm. We set this constant to $\epsilon=0.01$.

We shall compare several mixtures -- original Kies (A1), bimodal (A2), multimodal (A3), binomial (A4), geometric (A5), exponential (A6), gamma (A7), and beta (A8). We assume for the last five mixtures that the random variable $\lambda$ does not exhibit the corresponding distribution, but its linear transformation. This increases significantly the applicability of the proposed models. If the original random variable is denoted by $\xi$, then the MGF of the resulting one can be derived as $\psi_{a\xi+b}\left(x\right)=e^{bx}\psi_{\xi}\left(ax\right)$.  The PDF of the random variable $\lambda=a\xi+b$  can be obtained via Corollary \ref{MGF} -- see also Section \ref{examples}.  All resulting PDFs are reported in Appendix \ref{app1}. Note that if $b>0$, then condition \eqref{eq6_14_1} is satisfied because the term $\left(at+b\right)$ in the integral 

\begin{equation}\label{eq20_2}
\mathbb{E}=\left[\lambda^{-\frac{1}{\beta}}\right]=\int\limits_{\mathbb{R}}{\left(at+b\right)^{-\frac{1}{\beta}}f\left(t\right)dt}
\end{equation}

\noindent
does not influence the possible singularity of $f\left(t\right)$ in the zero. The case $b=0$ is considered in Section \ref{cont}.

\subsection{S\&P500 index}

The statistical sample consists of a total of  $10\ 717$ daily observations for the S\&P500 index in the period between January 2, 1980,  and July 01, 2022. We look for the market shocks  defined as the dates at which the index falls with more than two percent. More precisely, we are interested in the length of the periods between two shocks measured in working days -- the calm periods. Their number is 357 and they vary between $1$ and $950$. We divide them  by $1\ 000$ to fit the distribution's domain. For more details see     \cite{zaevski_kyurkchiev_2023} -- there can be found also the original Kies calibration. The interval $\left(0,1\right)$ is devised into $m=50$ sub-intervals for the current test. All mixture  estimations we prepare are reported in the first part of  Table  \ref{para_num}. The corresponding densities are presented in Figure \ref{sp_fig}. The inner figure is for  the distribution's core -- the interval $\left(0.03,0.2\right)$.  We can see that although the statistical data suggests that the initial density value is infinity, some of the estimated distributions do not exhibit this feature. In terms of Proposition \ref{prop_left} and Corollary \ref{cor_10} this means that the variable $\beta$, random or not, is larger than one. This is the case for the three-modal estimation as well as for the binomial, geometric, exponential, and gamma mixtures.

\subsection{Unemployment insurance issues}

The second example is related to the monthly observations of the unemployment insurance issues for the period between 1971 and 2018 -- a total of $574$ observations in the range  $\left[49\ 263,308\ 352\right]$. The data can be found at https://data.worlddatany-govns8zxewg or in \cite{Vasileva_Kyurkchiev}, pp. 162-164. Some statistical experiments based on the same data  are provided also  in \cite{ahmad2021new,he2020arcsine,zaevski_kyurkchiev_2024,zhenwu2021genesis}. We need first to process the statistical sample to make it convenient for our investigation.  In \cite{zaevski_kyurkchiev_2024} the same  statistical sample is divided by  $50\ 000$ and thus the new data is in the interval $\left[0.9853,6.1670\right]$. The results of this article strongly indicate that the estimated Kies style distributions are supported in  intervals close to $\left(1,9\right)$. This motivates us to transform the original data $S$ to

\begin{equation}\label{eq20_3}
S_{\mathrm{new}}=\frac{S-\min\left(S\right)}{1.5\left(\max\left(S\right)-\min\left(S\right)\right)}.
\end{equation}

\noindent
 This way we can compare the current results with those presented in \cite{zaevski_kyurkchiev_2024}.  We devise now  the interval $\left(0,1\right)$ into $m=20$ sub-intervals. The derived estimates   are reported  in the second part of Table  \ref{para_num} and the corresponding PDFs can be seen in Figure \ref{uii_fig}. We can observe that the best fit produces the multimodal distribution -- this is true for the first statistical sample too. This is not surprising since these distributions are very flexible to fit different curves. On the other hand, they are prone to over-fitting and thus they should only be used when there is a strong evidence that the considered  statistical sample is indeed multimodal. Also,  the fit which produces the exponential mixture is remarkable given that the exponential distribution is driven  by only one parameter. Of course, the gamma approximation is better since the gamma distribution generalizes the exponential one.

\section*{Acknowledgments}

The first author is financed by the European Union-NextGenerationEU, through the National Recovery and Resilience Plan of the Republic of Bulgaria, project No BG-RRP-2.004-0008.

The second author is financed by the European Union-NextGenerationEU, through the National Recovery and Resilience Plan of the Republic of Bulgaria, project No BG-RRP-2.004-0001-C01.

\section*{Declarations}
{\bf Conflict of interest} The authors declare no competing interests.

\begin{appendix}
\section{PDFs of the linear transformed distributions}\label{app1}

\begin{equation}\label{eqq_61_10}
\begin{split}
f_{\mathrm{bin}}\left(t\right)&=\beta\frac{t^{\beta-1}}{\left(1-t\right)^{\beta+1}}\exp\left(-b\left(\frac{t}{1-t}\right)^\beta\right)\left(1-p+p\exp\left(-a\left(\frac{t}{1-t}\right)^\beta\right)\right)^{n-1}\times\\
&\times\left[b-bp+p\left(na+b\right)\exp\left(-a\left(\frac{t}{1-t}\right)^\beta\right)\right]\\
f_{\mathrm{geo}}\left(t\right)&=p\beta t^{\beta-1}\frac{\left(b+a\right)\exp\left(\left(a-b\right)\left(\frac{t}{1-t}\right)^\beta\right)-b\left(1-p\right)\exp\left(-b\left(\frac{t}{1-t}\right)^\beta\right)}{\left(\exp\left(a\left(\frac{t}{1-t}\right)^\beta\right)+p-1\right)^2\left(1-t\right)^{\beta+1}}\\
f_{\mathrm{exp}}\left(t\right)&=\frac{\theta\beta t^{\beta-1}\left(1-t\right)^{\beta-1}\exp\left(-b\left(\frac{t}{1-t}\right)^\beta\right)}{\left(\theta\left(1-t\right)^{\beta}+at^{\beta}\right)^2}\left[b\left(\theta+a\left(\frac{t}{1-t}\right)^\beta\right)+a\right]\\
f_{\gamma}\left(t\right)&=\frac{\theta^\alpha\beta t^{\beta-1}\left(1-t\right)^{\alpha\beta-1}\exp\left(-b\left(\frac{t}{1-t}\right)^\beta\right)}{\left(\theta\left(1-t\right)^{\beta}+at^{\beta}\right)^{\alpha+1}}\left[b\left(\theta+a\left(\frac{t}{1-t}\right)^\beta\right)+a\alpha\right]\\
f_{\beta}\left(t\right)&=\beta\exp\left(-b\left(\frac{t}{1-t}\right)^\beta\right) \frac{t^{\beta-1}}{\left(1-t\right)^{\beta+1}}\\
&\times\left[\begin{array}{l}
b{}_1F_1\left(\alpha,\alpha+\theta,-a\left(\frac{t}{1-t}\right)^\beta\right)\\
+\frac{\alpha a}{\alpha+\theta}{}_1F_1\left(\alpha+1,\alpha+\theta+1,-a\left(\frac{t}{1-t}\right)^\beta\right)
\end{array}\right]
\end{split}
\end{equation}

\end{appendix}

\bibliographystyle{chicago}
\bibliography{Kies_bib}

\begin{thebibliography}{}

\bibitem[\protect\citeauthoryear{Abramowitz and Stegun}{Abramowitz and
  Stegun}{1968}]{abramowitz1968handbook}
Abramowitz, M. and I.~Stegun (1968).
\newblock {\em Handbook of mathematical functions with formulas, graphs, and
  mathematical tables}, Volume~55.
\newblock {US} Government printing office.

\bibitem[\protect\citeauthoryear{Afify, Gemeay, Alfaer, Cordeiro, and
  Hafez}{Afify et~al.}{2022}]{afify2022power}
Afify, A., A.~Gemeay, N.~Alfaer, G.~Cordeiro, and E.~Hafez (2022).
\newblock Power-modified {Kies}-exponential distribution: Properties, classical
  and {Bayesian} inference with an application to engineering data.
\newblock {\em Entropy\/}~{\em 24\/}(7), 883.

\bibitem[\protect\citeauthoryear{Ahmad, Mahmoudi, Roozegar, Alizadeh, and
  Afify}{Ahmad et~al.}{2021}]{ahmad2021new}
Ahmad, Z., E.~Mahmoudi, R.~Roozegar, M.~Alizadeh, and A.~Afify (2021).
\newblock A new exponential-{X} family: modeling extreme value data in the
  finance sector.
\newblock {\em Mathematical Problems in Engineering\/}~{\em 2021}, 1--14.

\bibitem[\protect\citeauthoryear{Al-Babtain, Shakhatreh, Nassar, and
  Afify}{Al-Babtain et~al.}{2020}]{al2020new}
Al-Babtain, A., M.~Shakhatreh, M.~Nassar, and A.~Afify (2020).
\newblock A new modified {Kies} family: Properties, estimation under complete
  and type-{II} censored samples, and engineering applications.
\newblock {\em Mathematics\/}~{\em 8\/}(8), 1345.

\bibitem[\protect\citeauthoryear{Alsubie}{Alsubie}{2021}]{alsubie2021properties}
Alsubie, A. (2021).
\newblock Properties and applications of the modified {Kies--Lomax}
  distribution with estimation methods.
\newblock {\em Journal of Mathematics\/}~{\em 2021}, 1--18.

\bibitem[\protect\citeauthoryear{Blasques, {van Brummelen}, Gorgi, and
  Koopman}{Blasques et~al.}{2024}]{BLASQUES2024105575}
Blasques, F., J.~{van Brummelen}, P.~Gorgi, and S.~Koopman (2024).
\newblock Maximum likelihood estimation for non-stationary location models with
  mixture of normal distributions.
\newblock {\em Journal of Econometrics\/}~{\em 238\/}(1), 105575.

\bibitem[\protect\citeauthoryear{D’Amico, {De Blasis}, and Petroni}{D’Amico
  et~al.}{2023}]{DAMICO2023129335}
D’Amico, G., R.~{De Blasis}, and F.~Petroni (2023).
\newblock The mixture transition distribution approach to networks: Evidence
  from stock markets.
\newblock {\em Physica A: Statistical Mechanics and its Applications\/}~{\em
  632}, 129335.

\bibitem[\protect\citeauthoryear{{dos Santos}, {do Nascimento}, {da Silva
  Jale}, Xavier, and Ferreira}{{dos Santos} et~al.}{2024}]{DOSSANTOS2024113990}
{dos Santos}, F., K.~{do Nascimento}, J.~{da Silva Jale}, S.~Xavier, and
  T.~Ferreira (2024).
\newblock Brazilian wind energy generation potential using mixtures of
  {Weibull} distributions.
\newblock {\em Renewable and Sustainable Energy Reviews\/}~{\em 189}, 113990.

\bibitem[\protect\citeauthoryear{Gradshteyn and Ryzhik}{Gradshteyn and
  Ryzhik}{2014}]{gradshteyn2014table}
Gradshteyn, I. and I.~Ryzhik (2014).
\newblock {\em Table of integrals, series, and products}.
\newblock Academic press.

\bibitem[\protect\citeauthoryear{Hashempour}{Hashempour}{2022}]{hashempour2022weighted}
Hashempour, M. (2022).
\newblock A weighted {Topp-Leone G} family of distributions: properties,
  applications for modelling reliability data and different method of
  estimation.
\newblock {\em Hacettepe Journal of Mathematics and Statistics\/}~{\em
  51\/}(5), 1420--1441.

\bibitem[\protect\citeauthoryear{He, Ahmad, Afify, and Goual}{He
  et~al.}{2020}]{he2020arcsine}
He, W., Z.~Ahmad, A.~Afify, and H.~Goual (2020).
\newblock The arcsine exponentiated-{X} family: validation and insurance
  application.
\newblock {\em Complexity\/}~{\em 2020}, 1--18.

\bibitem[\protect\citeauthoryear{Hui, Li, Follmann, and Qin}{Hui
  et~al.}{2022}]{https://doi.org/10.1002/sim.9367}
Hui, Z., P.~Li, D.~Follmann, and J.~Qin (2022).
\newblock A mixture distribution approach for assessing genetic impact from
  twin study.
\newblock {\em Statistics in Medicine\/}~{\em 41\/}(14), 2513--2522.

\bibitem[\protect\citeauthoryear{Kies}{Kies}{1958}]{Kies1958}
Kies, J. (1958).
\newblock The strength of glass performance.
\newblock In {\em Naval Research Lab Report}, Volume 5093. Washington, D.C.

\bibitem[\protect\citeauthoryear{Kumar and Dharmaja}{Kumar and
  Dharmaja}{2014}]{satheesh2014some}
Kumar, C.~S. and S.~Dharmaja (2014).
\newblock On some properties of {Kies} distribution.
\newblock {\em Metron\/}~{\em 72\/}(1), 97--122.

\bibitem[\protect\citeauthoryear{Kumar and Dharmaja}{Kumar and
  Dharmaja}{2017a}]{kumar2017exponentiated}
Kumar, C.~S. and S.~Dharmaja (2017a).
\newblock The exponentiated reduced {Kies} distribution: Properties and
  applications.
\newblock {\em Communications in Statistics-Theory and Methods\/}~{\em
  46\/}(17), 8778--8790.

\bibitem[\protect\citeauthoryear{Kumar and Dharmaja}{Kumar and
  Dharmaja}{2017b}]{kumar2017modified}
Kumar, C.~S. and S.~Dharmaja (2017b).
\newblock On modified {Kies} distribution and its applications.
\newblock {\em Journal of Statistical Research\/}~{\em 51\/}(1), 41--60.

\bibitem[\protect\citeauthoryear{Li, Liu, Sun, and Cai}{Li
  et~al.}{2023}]{LI2023102018}
Li, S., Y.~Liu, Y.~Sun, and Y.~Cai (2023).
\newblock Deep learning-based channel estimation using gaussian mixture
  distribution and expectation maximum algorithm.
\newblock {\em Physical Communication\/}~{\em 58}, 102018.

\bibitem[\protect\citeauthoryear{Liu, Xie, Edwards, and Yu}{Liu
  et~al.}{2023}]{liu2023mixture}
Liu, Y., D.~Xie, D.~Edwards, and S.~Yu (2023).
\newblock Mixture copulas with discrete margins and their application to
  imbalanced data.
\newblock {\em Journal of the Korean Statistical Society\/}~{\em 52\/}(4),
  878--900.

\bibitem[\protect\citeauthoryear{Maiboroda, Miroshnychenko, and
  Sugakova}{Maiboroda et~al.}{2022}]{maiboroda_miroshnychenko_sugakova_2022}
Maiboroda, R., V.~Miroshnychenko, and O.~Sugakova (2022).
\newblock Jackknife for nonlinear estimating equations.
\newblock {\em Modern Stochastics: Theory and Applications\/}~{\em 9\/}(4),
  377--399.

\bibitem[\protect\citeauthoryear{Naderi and M.J.Nooghabi}{Naderi and
  M.J.Nooghabi}{2024}]{NADERI2024115433}
Naderi, M. and M.J.Nooghabi (2024).
\newblock Clustering asymmetrical data with outliers: Parsimonious mixtures of
  contaminated mean-mixture of normal distributions.
\newblock {\em Journal of Computational and Applied Mathematics\/}~{\em 437},
  115433.

\bibitem[\protect\citeauthoryear{Sanku, Nassarn, and Kumar}{Sanku
  et~al.}{2019}]{sanku2019moments}
Sanku, D., M.~Nassarn, and D.~Kumar (2019).
\newblock Moments and estimation of reduced {Kies} distribution based on
  progressive {type-II} right censored order statistics.
\newblock {\em Hacettepe Journal of Mathematics and Statistics\/}~{\em
  48\/}(1), 332--350.

\bibitem[\protect\citeauthoryear{Sendov}{Sendov}{1990}]{sendov1990hausdorff}
Sendov, B. (1990).
\newblock {\em Hausdorff approximations}, Volume~50.
\newblock Springer Science \& Business Media.

\bibitem[\protect\citeauthoryear{Shafiq, Çolak, Swarup, Sindhu, and
  Lone}{Shafiq et~al.}{2022}]{https://doi.org/10.1002/adts.202200100}
Shafiq, A., A.~Çolak, C.~Swarup, T.~Sindhu, and S.~Lone (2022).
\newblock Reliability analysis based on mixture of lindley distributions with
  artificial neural network.
\newblock {\em Advanced Theory and Simulations\/}~{\em 5\/}(8), 2200100.

\bibitem[\protect\citeauthoryear{Silveira, Gomes-Silva, de~Brito, Jale,
  de~Gusm{\~a}o, Xavier-Junior, and Rocha}{Silveira
  et~al.}{2023}]{silveira2023modelling}
Silveira, F., F.~Gomes-Silva, C.~de~Brito, J.~Jale, F.~de~Gusm{\~a}o,
  S.~Xavier-Junior, and J.~Rocha (2023).
\newblock Modelling wind speed with a univariate probability distribution
  depending on two baseline functions.
\newblock {\em Hacettepe Journal of Mathematics and Statistics\/}~{\em
  52\/}(3), 808–--827.

\bibitem[\protect\citeauthoryear{Sobhi}{Sobhi}{2021}]{al2021modified}
Sobhi, M.~A. (2021).
\newblock The modified {Kies--{F}r{\'e}chet} distribution: properties,
  inference and application.
\newblock {\em AIMS Math\/}~{\em 6}, 4691--4714.

\bibitem[\protect\citeauthoryear{Vasileva}{Vasileva}{2023}]{math11224620}
Vasileva, M. (2023).
\newblock On {Topp-Leone-G} power series: Saturation in the {Hausdorff} sense
  and applications.
\newblock {\em Mathematics\/}~{\em 11\/}(22), 4620.

\bibitem[\protect\citeauthoryear{Vasileva and Kyurkchiev}{Vasileva and
  Kyurkchiev}{2023}]{Vasileva_Kyurkchiev}
Vasileva, M. and N.~Kyurkchiev (2023).
\newblock {\em Insuarance Mathematics}.
\newblock Plovdiv University Press.
\newblock in Bulgarian.

\bibitem[\protect\citeauthoryear{Wang}{Wang}{2023}]{https://doi.org/10.1002/wics.1611}
Wang, H. (2023).
\newblock Tolerance limits for mixture-of-normal distributions with application
  to {COVID-19} data.
\newblock {\em WIREs Computational Statistics\/}~{\em 15\/}(6), e1611.

\bibitem[\protect\citeauthoryear{Wang, Meng, Zhang, Xia, Xia, and Li}{Wang
  et~al.}{2024}]{WANG2024246}
Wang, Y., Z.~Meng, Z.~Zhang, M.~Xia, L.~Xia, and W.~Li (2024).
\newblock A regularization algorithm of dynamic light scattering for estimating
  the particle size distribution of dual-substance mixture in water.
\newblock {\em Particuology\/}~{\em 89}, 246--257.

\bibitem[\protect\citeauthoryear{Yan, Guo, and Wang}{Yan
  et~al.}{2022}]{YAN2022493}
Yan, A., J.~Guo, and D.~Wang (2022).
\newblock Robust stochastic configuration networks for industrial data
  modelling with {Student’s}-t mixture distribution.
\newblock {\em Information Sciences\/}~{\em 607}, 493--505.

\bibitem[\protect\citeauthoryear{Yeleyko and Yarova}{Yeleyko and
  Yarova}{2022}]{yeleyko2022mixture}
Yeleyko, Y. and O.~Yarova (2022).
\newblock Mixture of distributions based on the {Markov} chain.
\newblock {\em Cybernetics and Systems Analysis\/}~{\em 58\/}(5), 754--757.

\bibitem[\protect\citeauthoryear{Zaevski and Kyurkchiev}{Zaevski and
  Kyurkchiev}{2022}]{kies_corrections}
Zaevski, T. and N.~Kyurkchiev (2022).
\newblock Some notes on the four-parameters {Kies} distribution.
\newblock {\em Comptes rendus de l'Acad\'emie bulgare des Sciences\/}~{\em
  75\/}(10), 1403--1409.

\bibitem[\protect\citeauthoryear{Zaevski and Kyurkchiev}{Zaevski and
  Kyurkchiev}{2023}]{zaevski_kyurkchiev_2023}
Zaevski, T. and N.~Kyurkchiev (2023).
\newblock On some composite {Kies} families: distributional properties and
  saturation in {Hausdorff} sense.
\newblock {\em Modern Stochastics: Theory and Applications\/}~{\em 10\/}(3),
  287--312.

\bibitem[\protect\citeauthoryear{Zaevski and Kyurkchiev}{Zaevski and
  Kyurkchiev}{2024a}]{zaevski_kyurkchiev_2024}
Zaevski, T. and N.~Kyurkchiev (2024a).
\newblock On min- and max-{Kies} families: distributional properties and
  saturation in {Hausdorff} sense.
\newblock {\em Modern Stochastics: Theory and Applications\/}~{\em 11\/}(1),
  1--24.

\bibitem[\protect\citeauthoryear{Zaevski and Kyurkchiev}{Zaevski and
  Kyurkchiev}{2024b}]{zaevski_kyurkchiev_trig_2024}
Zaevski, T. and N.~Kyurkchiev (2024b).
\newblock On the {Hausdorff} saturation of some trigonometric-{Kies} families.
\newblock {\em Palestine Journal of Mathematics\/}~{\em accepted paper}, --.

\bibitem[\protect\citeauthoryear{Zhang, Dong, and Feng}{Zhang
  et~al.}{2023}]{ZHANG2023110352}
Zhang, Y., Y.~Dong, and R.~Feng (2023).
\newblock Bayes-informed mixture distribution for the evd estimation and
  dynamic reliability analysis.
\newblock {\em Mechanical Systems and Signal Processing\/}~{\em 197}, 110352.

\bibitem[\protect\citeauthoryear{Zhenwu, Ahmad, Almaspoor, and Khosa}{Zhenwu
  et~al.}{2021}]{zhenwu2021genesis}
Zhenwu, Y., Z.~Ahmad, Z.~Almaspoor, and S.~Khosa (2021).
\newblock On the genesis of the {Marshall-Olkin} family of distributions via
  the {T-X} family approach: Statistical modeling.
\newblock {\em CMC-COMPUTERS MATERIALS \& CONTINUA\/}~{\em 67\/}(1), 753--760.

\end{thebibliography}

\section{Figures and tables}

\begin{figure}
    \caption{ Mixture-Kies PDFs}\label{PDF}
    \setcounter{subfigure}{0}
		          \subfloat[$\lambda\in\left\{0.1,2\right\},\beta=2$] {\label{peaks_2}\includegraphics[width=.50\textwidth]{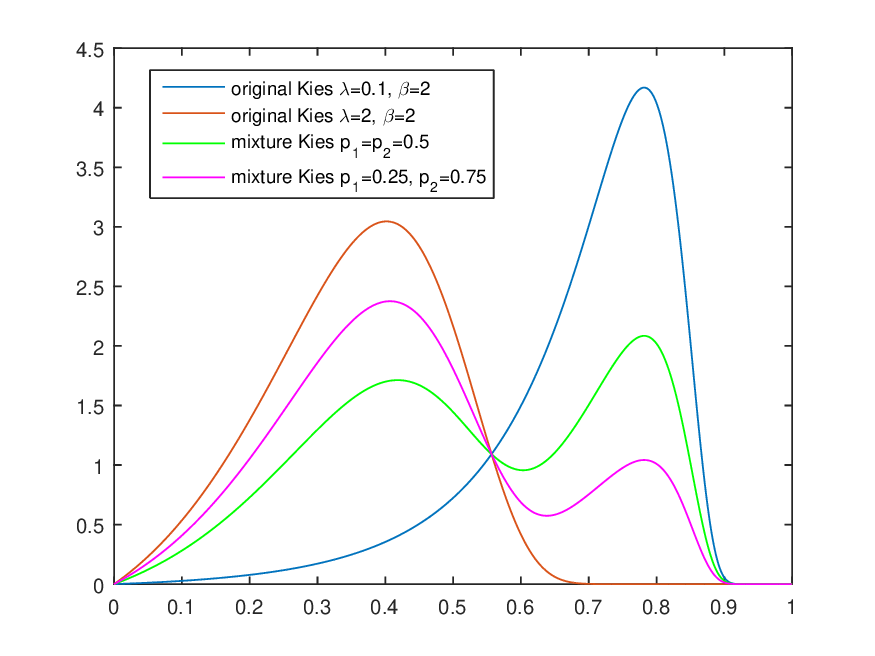}}\subfloat[$\lambda\in\left\{1,2\right\},\beta=2$] {\label{peaks_1}\includegraphics[width=.50\textwidth]{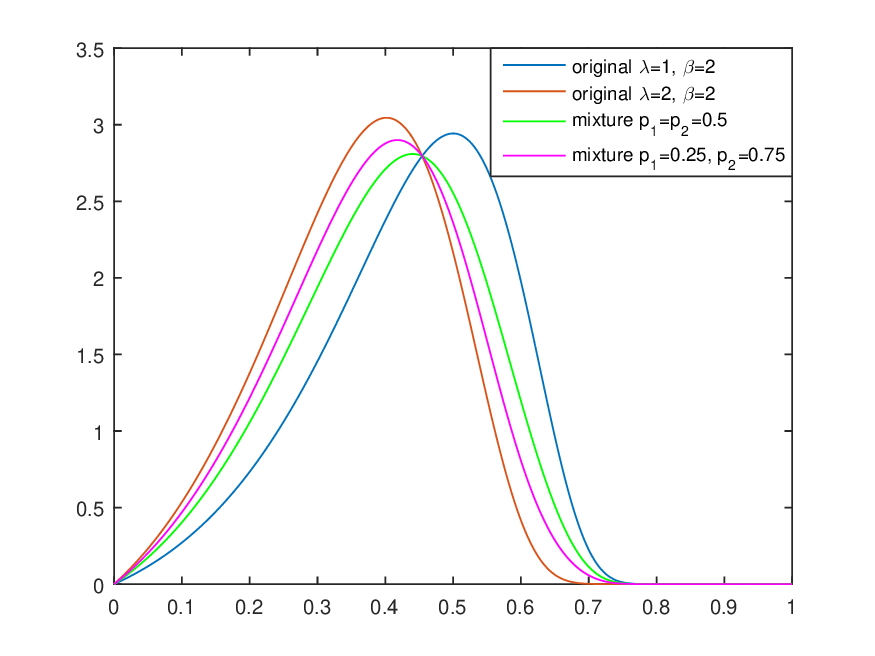}}\\
							\subfloat[$\lambda\in\left\{0.1,2\right\},\beta=2$] {\label{peaks_2_fin}\includegraphics[width=.50\textwidth]{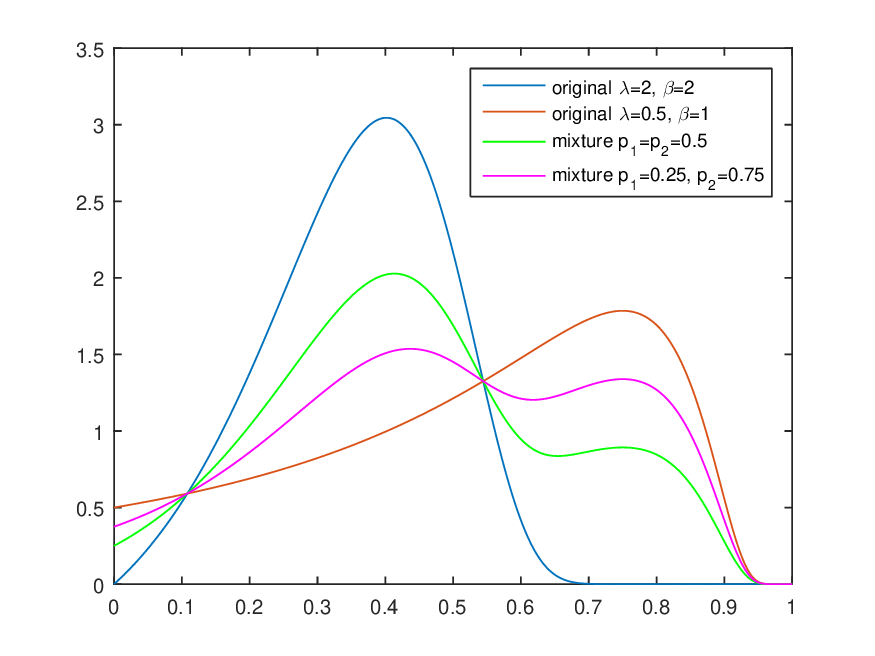}}		
						  \subfloat[$\lambda\in\left\{0.1,2\right\},\beta=2$] {\label{peaks_2_infin}\includegraphics[width=.50\textwidth]{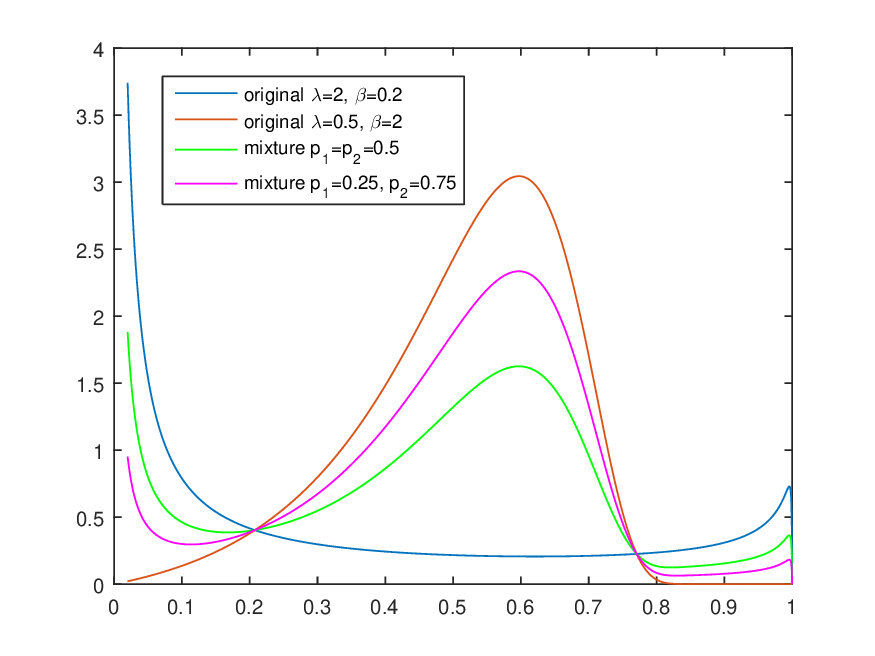}}\\
							\subfloat[$\lambda\in\left\{0.1,2\right\},\beta=2$] {\label{peaks_4}\includegraphics[width=.50\textwidth]{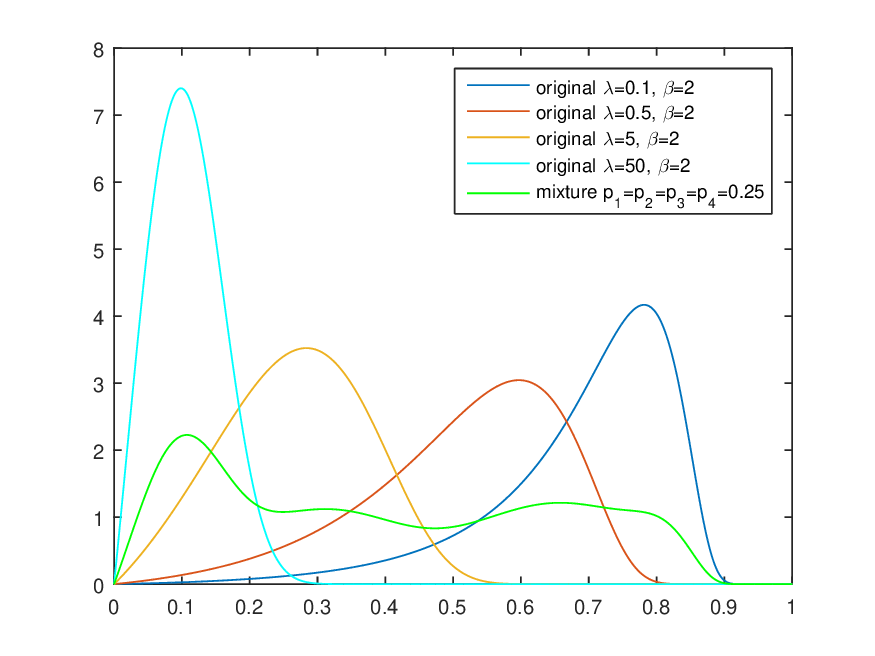}}
							\subfloat[Binomial mixture] {\label{binomial}\includegraphics[width=.50\textwidth]{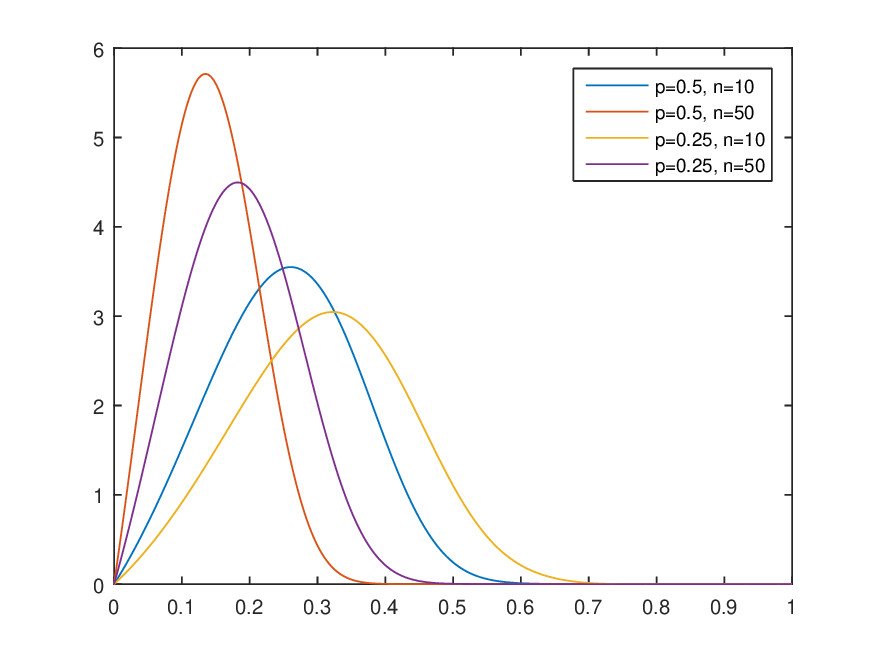}}
\end{figure}

\begin{figure}
    \caption{ Mixture-Kies PDFs}\label{PDF_2}
    \setcounter{subfigure}{0}
		          \subfloat[Geometric mixture] {\label{geometric}\includegraphics[width=.50\textwidth]{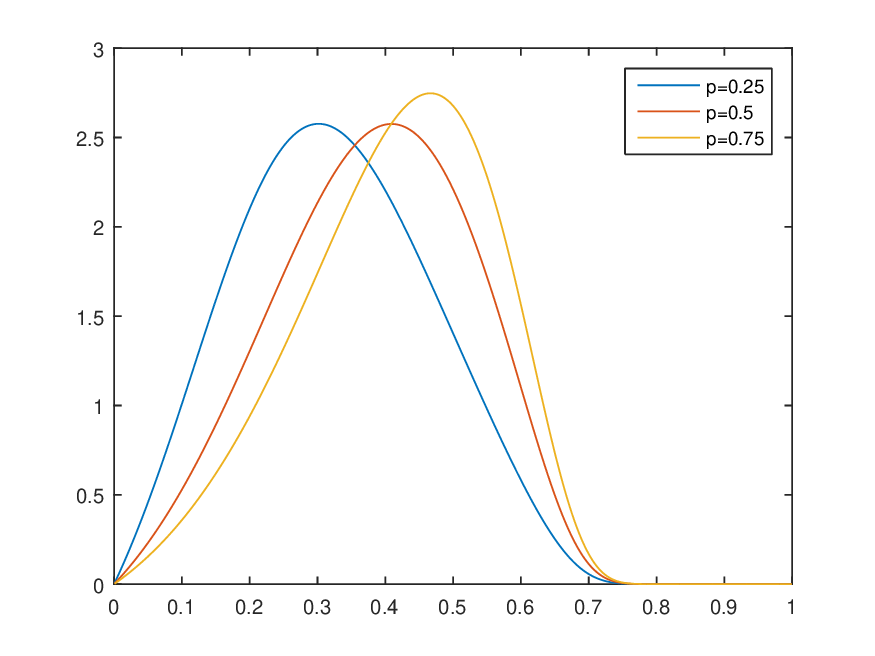}}
							\subfloat[Exponential mixture] {\label{expo}\includegraphics[width=.50\textwidth]{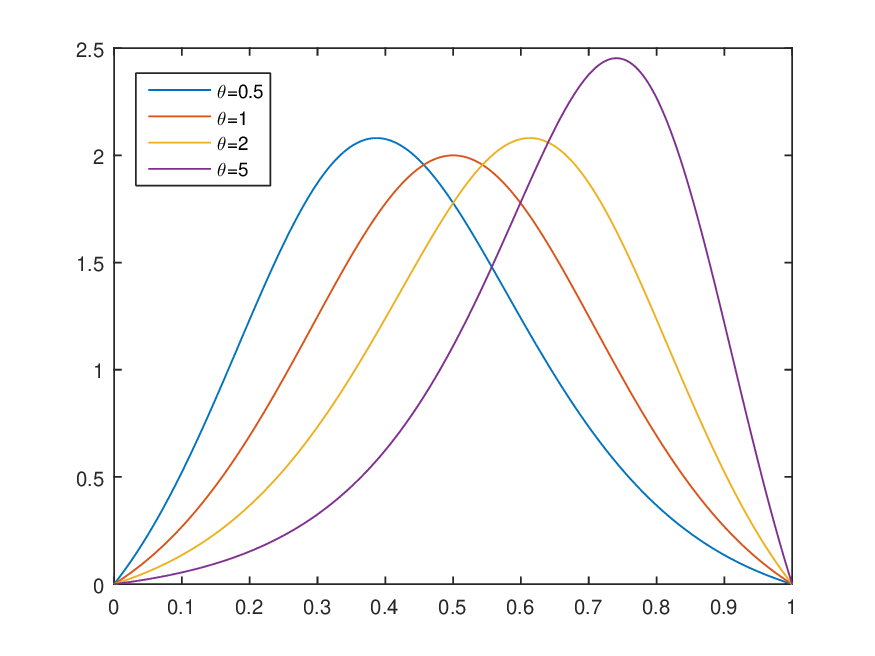}}\\
							\subfloat[Gamma mixture, $\beta=2$] {\label{gam_2}\includegraphics[width=.50\textwidth]{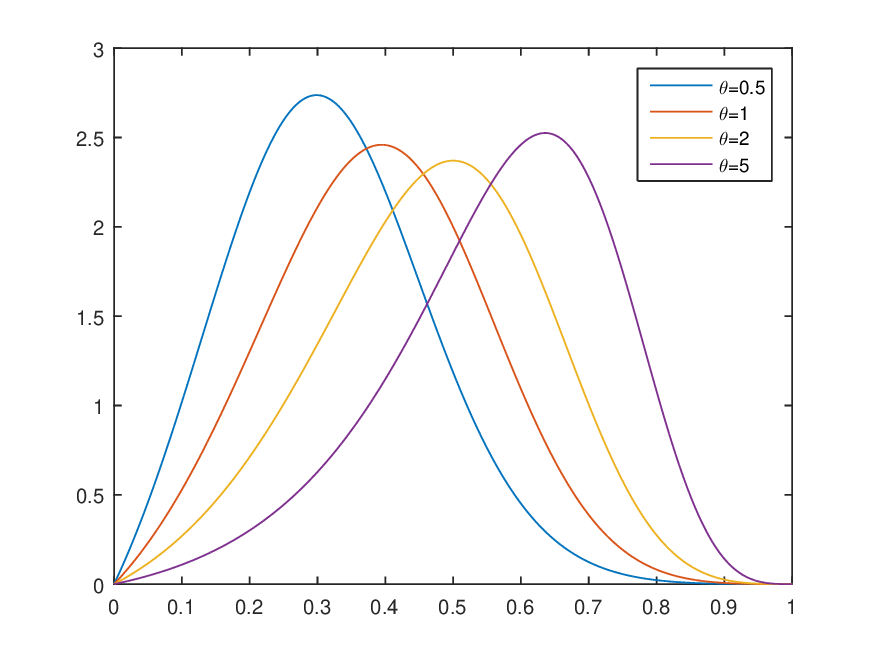}}
							\subfloat[Gamma mixture, $\beta=1$] {\label{gam_1}\includegraphics[width=.50\textwidth]{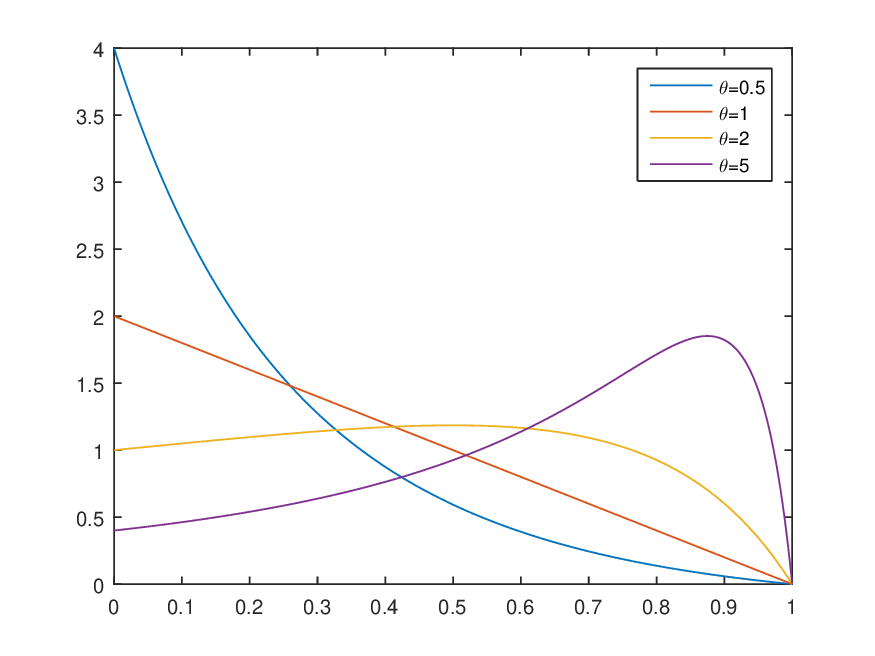}}\\
							\subfloat[Gamma mixture, $\beta=0.7$] {\label{gam_07}\includegraphics[width=.50\textwidth]{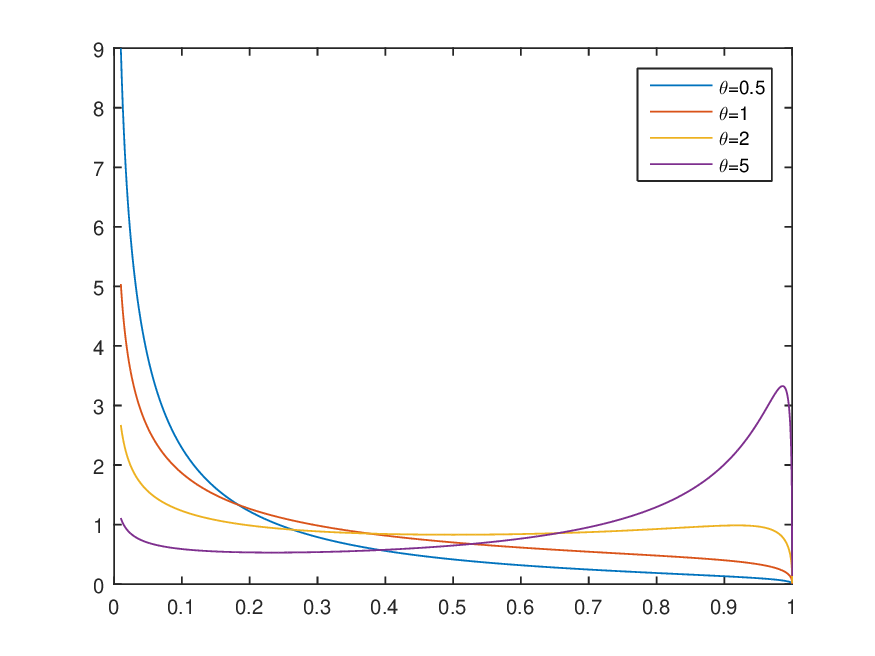}}
							\subfloat[Beta mixture] {\label{bet}\includegraphics[width=.50\textwidth]{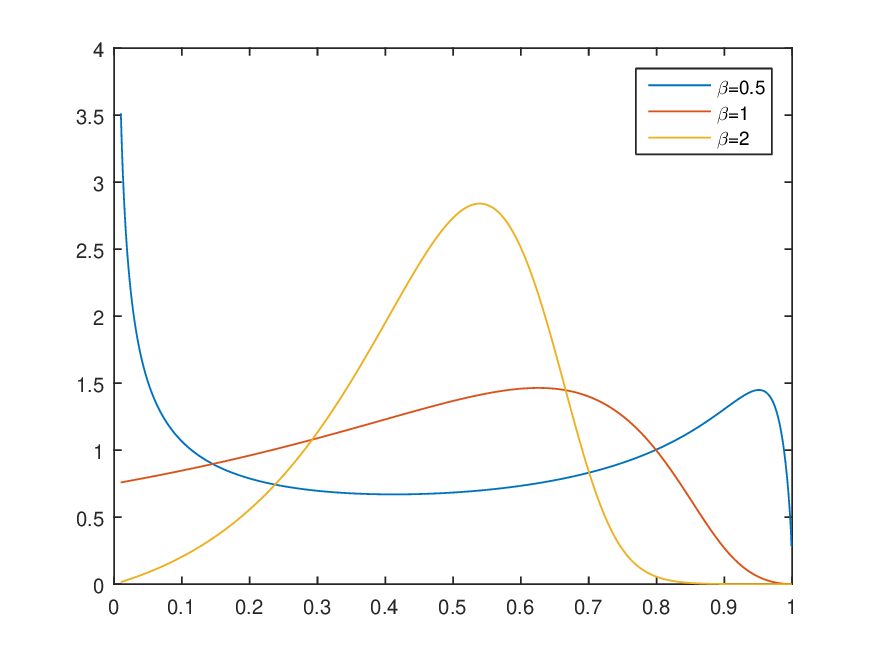}}
\end{figure}

\begin{figure}
    \caption{ Mixture-Kies CDFs with saturation}\label{CDF_2}
    \setcounter{subfigure}{0}
		          \subfloat[Bimodal mixture] {\label{bim_cdf}\includegraphics[width=.50\textwidth]{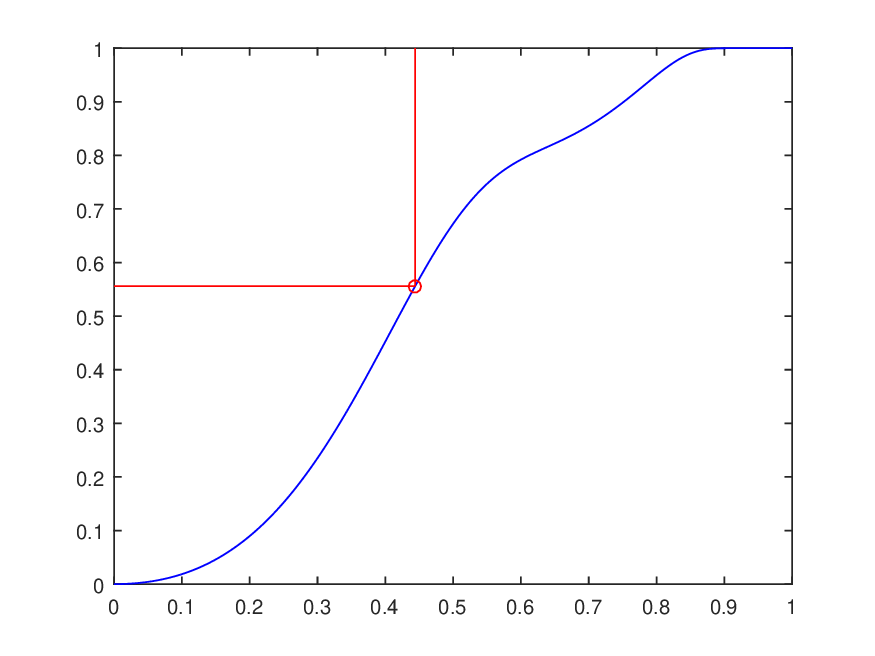}}
							\subfloat[Binomial mixture] {\label{bin_cdf}\includegraphics[width=.50\textwidth]{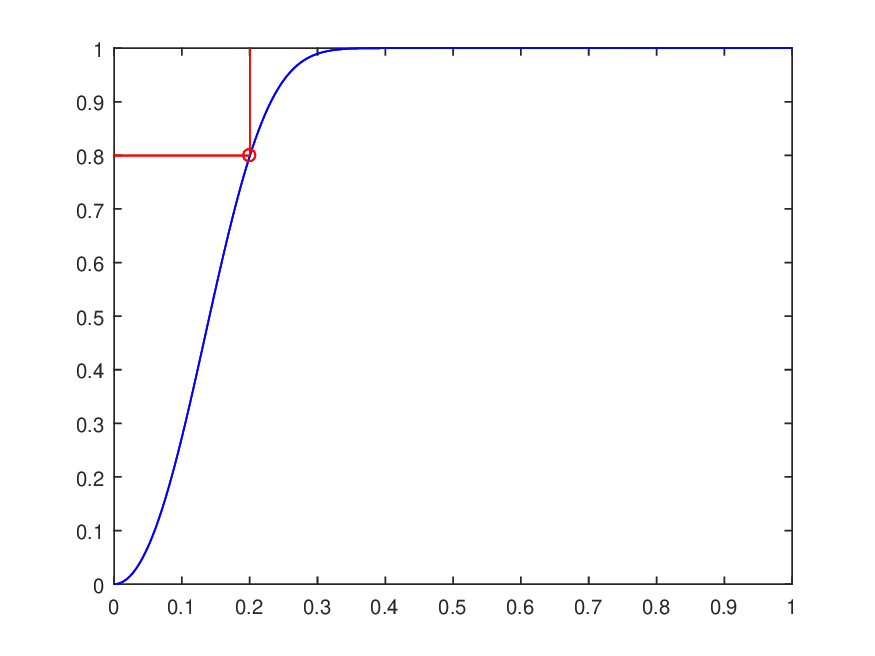}}\\
							\subfloat[Geometrical mixture] {\label{geom_cdf}\includegraphics[width=.50\textwidth]{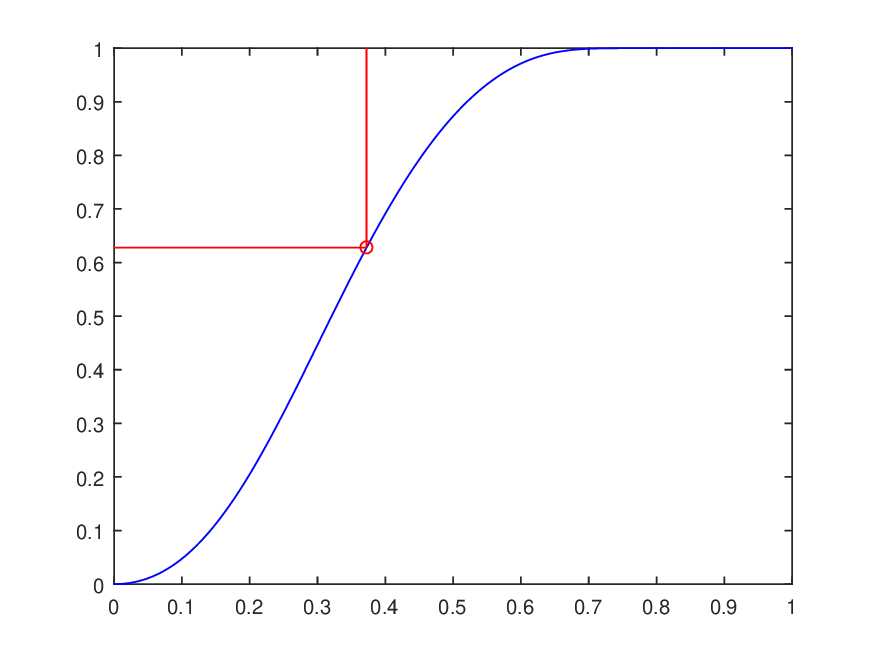}}
							\subfloat[Exponential mixture, $\beta=1$] {\label{exp_cdf}\includegraphics[width=.50\textwidth]{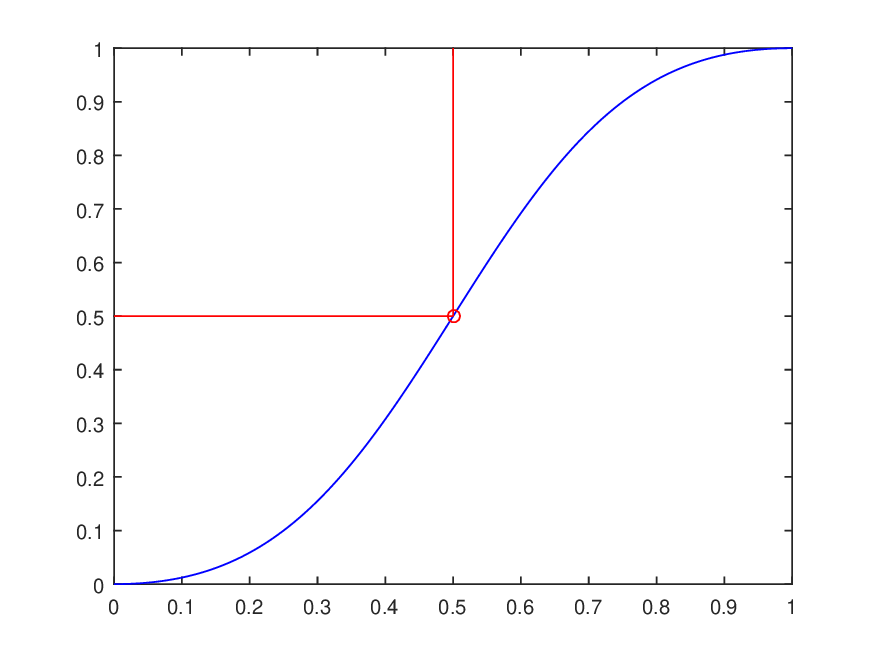}}\\
							\subfloat[Gamma mixture] {\label{gam_cdf}\includegraphics[width=.50\textwidth]{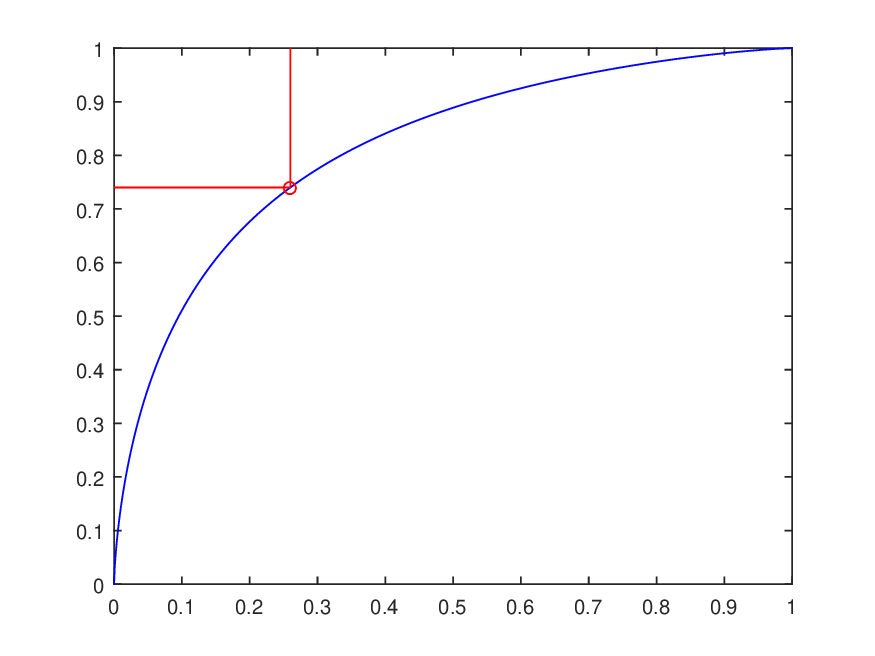}}
							\subfloat[Beta mixture] {\label{bet_cdf}\includegraphics[width=.50\textwidth]{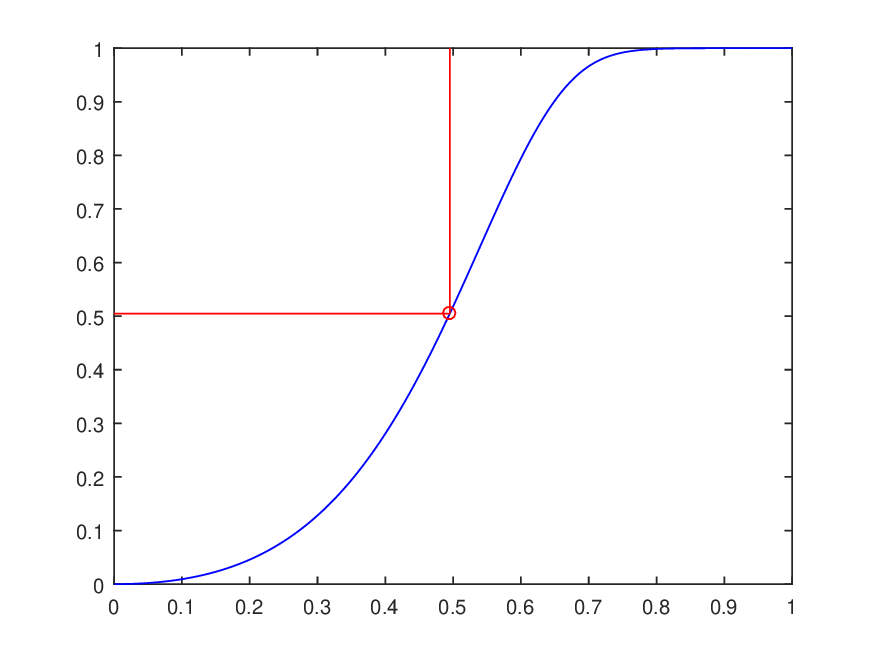}}
\end{figure}

\begin{figure}
    \caption{ Estimations}\label{est}
    \setcounter{subfigure}{0}
		          \subfloat[S\&P 500 index] {\label{sp_fig}\includegraphics[width=1\textwidth]{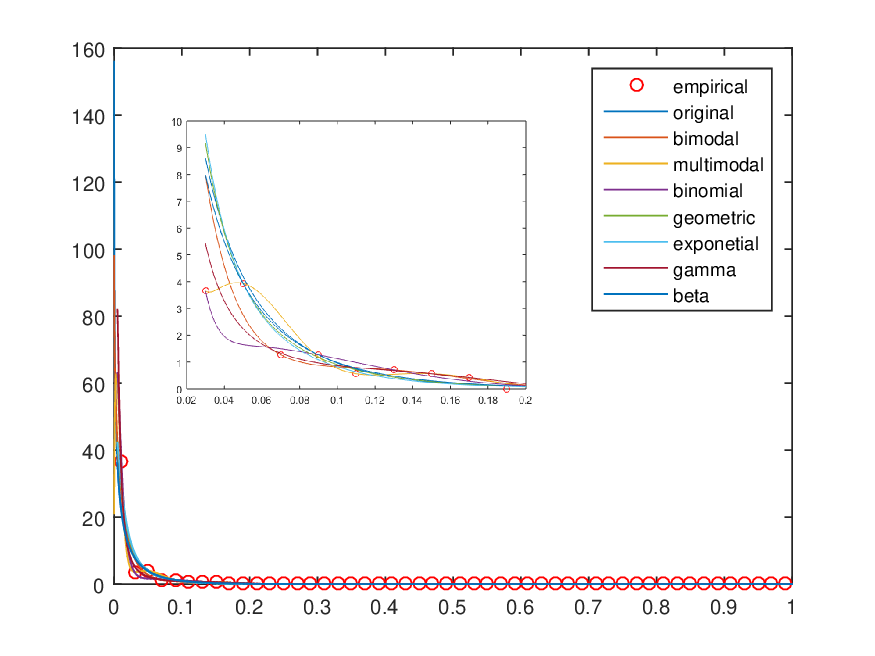}}\\
							\subfloat[Unemployment Insurance Issues] {\label{uii_fig}\includegraphics[width=1\textwidth]{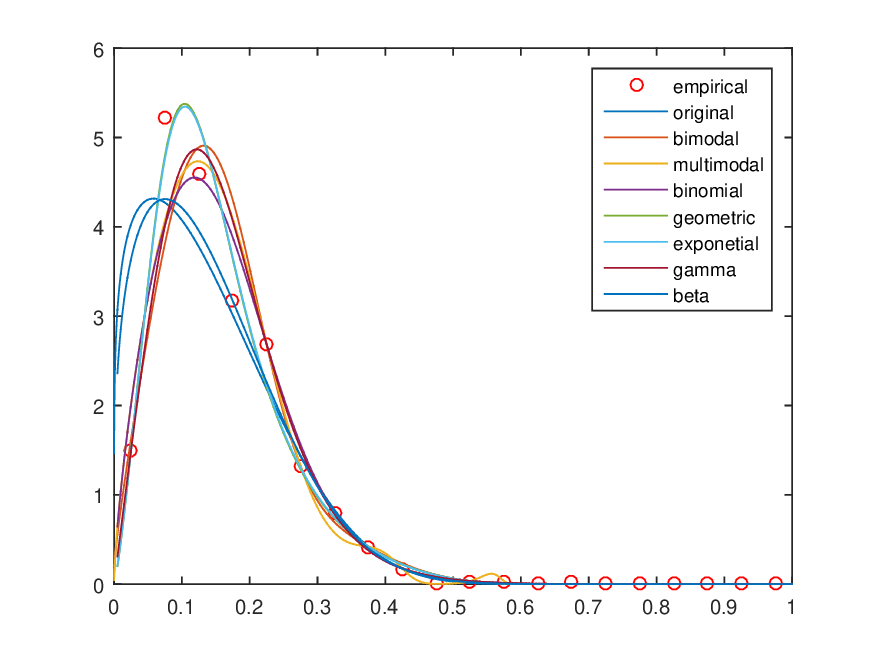}}\\
\end{figure}

\begin{sidewaystable}
\caption{Numerical Estimations}
  \label{para_num}
 \begin{tabular}{c|ccc|c|cccccc}

  \noalign{\smallskip}\hline\noalign{\smallskip}

	S\&P 500    & original& bimodal   & multimodal& model    & binomial &geometric &exponential& gamma                 & beta        \\
	
	\noalign{\smallskip}\hline\noalign{\smallskip}
	
              &  (A1)    &   (A2)    & (A3)     &          &(A4)      &  (A5)    &   (A6)   & (A7)                  & (A8)       \\ 
	
	  \noalign{\smallskip}\hline\noalign{\smallskip}
	$\lambda_1$ &$15.7857$&$40.6027$  &$735.5371$ &$ \beta  $& $1.4975$ & $1.1969$ &  $1.2479$ &$3.2018$               & $0.6653$        \\
  $\lambda_2$ & ---     &$109.2334$ &$920.2421$ & $ a     $&$560.3198$& $20.8690$& $74.2504$ &$732.5631$             & $6.4268$  \\
  $\lambda_3$ & ---     &  ---      &$576.0060$ &$ b     $ &$33.0044$ & $0.0000$ & $14.1627$ &$185.3902$             & $7.2682$    \\   
	$  \beta_1$ &$0.7120$ &$0.8796$   &$1.3843$   &$    p  $ &  $0.0017$& $0.1323$ &   ---     & ---                   & ---           \\
  $  \beta_2$ & ---     &$2.4893$   & $3.9519$  &$    n  $ &  $981$   & ---      &   ---     & ---                   & ---         \\
  $  \beta_3$ & ---     &  ---      & $2.2741$  &$\theta $ & ---      & ---      & $0.4193$  &$1.6238\times 10^{-5}$ & $4.0301\times 10^{-5}$          \\
	$      p_1$ & $1$     &$0.9165$   & $0.7175$  &$ \alpha $& ---      & ---      & ---       &$0.2529$             & $7.3923$   \\
	$      p_2$ & ---     &$0.0835$   & $0.0494$  & ---      & ---      &  ---     & ---       &  ---       ---      & ---   \\  
	$      p_3$ & ---     &  ---      & $0.2331$  & ---      & ---      &  ---     & ---       &   ---       ---     &  ---   \\ 
	
		  \noalign{\smallskip}\hline\noalign{\smallskip}
  
error         &$25.3497$&$22.8935$  &$21.6932$  &          &$23.0744$ & $24.7371$& $25.1232$ &$22.4339$            &$24.5851$     \\	

	\noalign{\smallskip}\hline\noalign{\smallskip}
	
UII          &  (A1)    &   (A2)    & (A3)     &          &(A4)      &  (A5)     &   (A6)    & (A7)                     & (A8)      \\    
	
	  \noalign{\smallskip}\hline\noalign{\smallskip}
	$\lambda_1$ &$7.0550$ &$27.4484$  &$0.0245  $ &$ \beta  $& $1.6951$   & $2.2257$ &  $2.2161$ &$1.9603$               & $1.2819$  \\
  $\lambda_2$ & ---     &$5.0320  $ &$14.9210$  & $ a     $&$3.8136   $ & $1.8680 $& $4.7886$  &$6.9343  $             & $3.9552$  \\
  $\lambda_3$ & ---     &  ---      &$16.9510 $ &$ b     $ &$2.8160   $ & $0.9654$ & $2.0540$  &$0.3263  $             & $4.9723$  \\  
	$  \beta_1$ &$1.1895$ &$1.9935$   &$16.3815$  &$    p  $ &  $0.0334  $& $0.0322$ &   ---     & ---                   & ---           \\
  $  \beta_2$ & ---     &$1.3558 $  & $1.6827$  &$    n  $ &  $95   $   & ---      &   ---     & ---                   & ---         \\
  $  \beta_3$ & ---     &  ---      & $6.0798$  &$\theta $ & ---        & ---      & $0.0869$  &$0.6134              $ & $0.0003$          \\
	$      p_1$ & $1$     &$0.6464$   & $0.0048$  &$ \alpha $& ---        & ---      & ---       &$2.1985$               & $0.0011$   \\
	$      p_2$ & ---     &$0.3536$   & $0.9584$  & ---      & ---        &  ---     & ---       &                       & ---         \\  
	$      p_3$ & ---     &  ---      & $0.0369$  & ---      & ---        &  ---     & ---       &                       & ---          \\ 
	
		  \noalign{\smallskip}\hline\noalign{\smallskip}
  
error         &$7.1362 $&$6.1300 $  &$2.8820 $  &          & $5.6375$   & $5.8140$ &$5.7291 $ &  $5.4682$               &$6.6078$     \\	

		  \noalign{\smallskip}\hline\noalign{\smallskip}					
  \end{tabular}  
\end{sidewaystable}

\end{document}